\documentclass[12pt,reqno]{amsart}
\usepackage{amsmath,amsthm,amssymb,amsfonts,amscd}
\usepackage{mathrsfs}
\usepackage{bbm}
\usepackage{bbding}

\usepackage{bm}
\usepackage{hyperref}

\usepackage{geometry}
\usepackage{color}
\usepackage{xcolor}

\usepackage{picture,epic}
\usepackage{tikz}

\numberwithin{equation}{section}

\theoremstyle{plain}
\newtheorem{theorem}{Theorem}[section]
\newtheorem{lemma}[theorem]{Lemma}

\newtheorem{proposition}[theorem]{Proposition}

\theoremstyle{definition}

\theoremstyle{remark}
\newtheorem{remark}[theorem]{Remark}

\renewcommand{\Re}{\operatorname{Re}}

\newcommand{\sgn}{\operatorname{sgn}}

\newcommand{\supp}{\operatorname{supp}}

\newcommand{\GL}{\operatorname{GL}}
\newcommand{\SL}{\operatorname{SL}}

\renewcommand{\mod}{\operatorname{mod}\ }

\newcommand{\dd}{\mathrm{d}}

\makeatletter
\def\@tocline#1#2#3#4#5#6#7{\relax
  \ifnum #1>\c@tocdepth 
  \else
    \par \addpenalty\@secpenalty\addvspace{#2}%
    \begingroup \hyphenpenalty\@M
    \@ifempty{#4}{%
      \@tempdima\csname r@tocindent\number#1\endcsname\relax
    }{%
      \@tempdima#4\relax
    }%
    \parindent\z@ \leftskip#3\relax \advance\leftskip\@tempdima\relax
    \rightskip\@pnumwidth plus4em \parfillskip-\@pnumwidth
    #5\leavevmode\hskip-\@tempdima
      \ifcase #1
       \or\or \hskip 1em \or \hskip 2em \else \hskip 3em \fi%
      #6\nobreak\relax
    \hfill\hbox to\@pnumwidth{\@tocpagenum{#7}}\par
    \nobreak
    \endgroup
  \fi}
\makeatother

\begin{document}

\title[Hybrid subconvexity bounds for twists of $\GL(3)\times \GL(2)$ $L$-functions]
{Hybrid subconvexity bounds for twists of \\ $\GL(3)\times \GL(2)$ $L$-functions}

\author{Bingrong Huang and Zhao Xu}

\address{Data Science Institute and School of Mathematics \\ Shandong University \\ Jinan \\ Shandong 250100 \\China}
\email{brhuang@sdu.edu.cn}

\address{School of Mathematics \\ Shandong University \\ Jinan \\ Shandong 250100 \\China}
\email{zxu@sdu.edu.cn}


\date{\today}

\begin{abstract}
  In this paper, we solve the hybrid subconvexity problem for $\GL(3)\times \GL(2)$ $L$-functions twisted by a primtive Dirichlet charater modulo $M$ (prime) in the $M$- and $t$-aspects. We also improve hybrid subconvexity bounds for twists of $\GL(3)$ $L$-functions in the $M$- and $t$-aspects.
\end{abstract}



\thanks{This work was in part supported by  the National Key Research and Development Program of China  2021YFA1000700 and  NSFC  12031008.
B.H. was also supported by NSFC 12001314 and the Young Taishan Scholars Program.
Z.X. was aslo supported by Natural Science Foundation of Shandong Province ZR2019MA011.
}

\maketitle

\section{Introduction} \label{sec:Intr}

The subconvexity problem of automorphic $L$-functions on the critical line is one of the central problems in number theory.
In general, let $\mathcal{C}$ denote the analytic conductor of the relevant $L$-function,
then one hopes to obtain a subconvexity bound $\mathcal{C}^{1/4-\delta}$ for some $\delta>0$
on the critical line $\Re s=1/2$.
Subconvexity bounds have many very important applications such as the equidistribution problems. For the $\GL(1)$ case, i.e., the Riemann zeta function and Dirichlet $L$-functions, subconvexity bounds are known for a long time thanks to Weyl \cite{Weyl} and Burgess \cite{Burgess}. For the last decades, many cases of $\GL(2)$ $L$-functions have been treated (see Michel--Venkatesh \cite{MichelVenkatesh} and the references therein).
In the recent ten years, people have made progress on $\GL(3)$ $L$-functions (see \cite{li2011bounds,blomer2012subconvexity,Munshi2015circleIII,Munshi2015circleIV,Munshi2018,Sharma2019} etc.).
In this paper, we extend the techniques to prove, for the first time, hybrid subconvexity bounds for $\GL(3)\times \GL(2)$ $L$-functions twisted by a primtive Dirichlet charater modulo $M$ (prime), which generalizes the best known bounds in the $M$- and $t$-aspects simultaneously. Our method also improves hybrid subconvexity bounds for twists of $\GL(3)$ $L$-functions due to \cite{Huang2019, Lin}.


Let $\pi$ be a Hecke--Maass cusp form of type $(\nu_1,\nu_2)$ for $\SL(3,\mathbb{Z})$
with the normalized Fourier coefficients $A(m,n)$.
The $L$-function of $\pi$ is defined as
\[ L(s,\pi) = \sum_{n\geq1} \frac{A(1,n)}{n^s}, \quad \Re(s)>1. \]
Let $f$ be a Hecke--Maass cusp form with the spectral parameter $t_f$ for $\SL(2,\mathbb{Z})$, with the normalized Fourier coefficients $\lambda_f(n)$.
The $L$-function of $f$ is defined by
\[
  L(s,f) = \sum_{n\geq1} \frac{\lambda_f(n)}{n^s}, \quad \Re(s)>1.
\]
Let $\chi$ be a primitive Dirichlet character modulo $M$.
The $\GL(3)\times \GL(2)\times \GL(1)$  Rankin--Selberg $L$-function is defined as
\[
  L(s,\pi\times f \times \chi) = \sum_{m\geq1}\sum_{n\geq1} \frac{A(m,n)\lambda_f(n)\chi(m^2n)}{(m^2n)^s}, \quad \Re(s)>1.
\]
Those $L$-functions have analytic continuation to the whole complex plane.
In this paper, we consider the $L$-values at the  point $1/2+it$.
The Phragm\'en--Lindel\"of principle implies the convexity bounds
\[
  L(1/2+it,\pi \times f\times \chi) \ll_{\pi,f,\varepsilon} (M(1+|t|))^{3/2+\varepsilon}.
\]
It is konwn
that the Riemann hypothesis for $L(s,\pi \times f\times \chi)$ implies the Lindel\"of hypothesis, i.e.,
$L(1/2+iT,\pi \times f\times \chi) \ll_{\pi,f,\varepsilon} (M(1+|t|))^{\varepsilon}.$
For $M=1$, the first subconvex exponent in $t$-aspect was obtained by Munshi \cite{Munshi2018}.
Recently, Lin--Sun \cite{LinSun} stated that
  \[
    L(1/2+it,\pi\times f)
    \ll_{\pi,f,\varepsilon}  (1+|t|)^{3/2-3/20+\varepsilon}.
  \]
For $t=0$ and prime $M$, it is proved in Sharma \cite{Sharma2019} that
  \[
    L(1/2,\pi\times f\times \chi)
    \ll_{\pi,f,\varepsilon}  M^{3/2-1/16+\varepsilon}.
  \]
In the context of $L$-functions, obtaining hybrid bounds that perfectly combine the two aspects is a difficult problem.
Our main result in this paper is the following subconvexity bounds.

\begin{theorem}\label{thm:subconv3*2}
  With the notation as above. Let $t\in\mathbb{R}$ and $M$ be prime. Then we have
  \[
    L(1/2+it,\pi\times f\times \chi)
    \ll_{\pi,f,\varepsilon}  M^{3/2-1/16+\varepsilon} (1+|t|)^{3/2-3/20+\varepsilon}.
  \]
\end{theorem}

\begin{remark}
  Below we will carry out the proof under the assumption $t\geq M^\varepsilon$ for some small $\varepsilon>0$. For the case $t\ll M^\varepsilon$, one can extend the method of Sharma \cite{Sharma2019} to prove $L(1/2+it,\pi\times f\times \chi)
    \ll_{t,\pi,f,\varepsilon}  M^{3/2-1/16+\varepsilon}$ with polynomial dependence on $t$.
  For the case $t\leq -M^\varepsilon$, the same result follows from the case $t\geq M^\varepsilon$ by the functional equation.
\end{remark}

\begin{remark}
  Let $\pi$, $\chi$ and $t$ be the same as above and $f$ be a weight $k$ Hecke modular form for $\SL(2,\mathbb{Z})$. The same hybrid subconvexity bounds for $L(1/2+it,\pi\times f\times \chi)$ can be proved by our method.
\end{remark}

Note that by the Hecke relation of the Fourier coefficients (see Goldfeld \cite[Theorem 6.4.11]{goldfeld2006automorphic}), we have
\[
  A(1,m)A(1,n) = \sum_{d\mid (m,n)} A\left(d,\frac{mn}{d^2}\right).
\]
Hence we have
\[
  L(s,\pi\times \chi)^2 = \sum_{m\geq1}\sum_{n\geq1} \frac{A(m,n)\tau(n) \chi(m^2n)}{(m^2n)^s}, \quad \Re(s)>1,
\]
where $\tau(n)=\sum_{d\mid n}1$ is the divisor function which
is the coefficient of the Eisenstein series for $\SL(2,\mathbb{Z})$.
The subconvexity bounds for $L(1/2+it,\pi\times\chi)$ follow from bounds for $L(1/2+it,\pi\times f\times \chi)$ with $f$ being a $\GL(2)$ Eisenstein series.

\begin{theorem}\label{thm:subconv3}
  With the notation as above. Let $t\in\mathbb{R}$ and $M$ be prime. Then we have
  \[
    L(1/2+it,\pi\times \chi) \ll_{\pi,\varepsilon} M^{3/4-1/32+\varepsilon} (1+|t|)^{3/4-3/40+\varepsilon}.
  \]
\end{theorem}

\begin{remark}
  The only difference in the proofs of Theorem \ref{thm:subconv3} and Theorem \ref{thm:subconv3*2} is that we need to use the Voronoi summation formula for $\tau(n)$ instead of those for Fourier coefficients of a $\GL(2)$ cusp form. This will give us another zero frequency contribution in the dual sum. This contribution will not have any effect on the final result. Indeed, in the generic case, the weight function for the sum of $\tau(n)$ is oscillating. By integration by parts, we can show its contribution is negligibly small.
\end{remark}

This improves the hybrid subconvexity bounds for twists of $\GL(3)$ $L$-functions due to the first author \cite{Huang2019} and Lin \cite{Lin}.  Recall that under the same assumptions Lin \cite{Lin} proved that
\[
    L(1/2+it,\pi\times \chi) \ll_{\pi,\varepsilon} \big(M (1+|t|)\big)^{3/4-1/36+\varepsilon}.
\]
One may give a quick comparison with Lin's work \cite{Lin}.
Actually, we have a different structure from Lin's paper.
As we said above, Theorem \ref{thm:subconv3} can be viewed as a subconvexity result for
$\GL(3)\times \GL(2)\times \GL(1)$ $L$-functions, where the $\GL(2)$-item is the Eisenstein series.
Lin's work is to consider the $L(\frac{1}{2}+it,\pi\times \chi)$ directly.

Heath-Brown \cite{Heath-Brown} proved the first hybrid subconvexity bounds for Dirichlet $L$-functions by extending the Burgess method and van der Corput method to give good estimates for hybrid sums $\sum \chi(n) n^{it}$. Recently, Petrow--Young \cite{PetrowYoung,PetrowYoung2019} proved the Weyl bound in both aspects by estimating moments of $L$-functions.
For the $\GL(2)$ case, Blomer--Harcos \cite{BlomerHarcos} proved the first hybrid subconvexity bounds in the $M$- and $t$-aspects by using moments of $L$-functions. Recently,
Fan--Sun \cite{FanSun} improved the bounds by  using a delta method.
Our method can also provide hybrid subconvexity bounds in the $\GL(1)$ and $\GL(2)$ settings, but are weaker than the best known results.


Our basic observation is that the subconvexity bounds for $\GL(3)\times \GL(2)\times \GL(1)$ $L$-functions in both $M$-aspect and $t$-aspect were proved by applying the Duke--Friedlander--Iwaniec delta method to separate oscillatory factors.
This suggests to us that in order to prove a hybrid subconvexity bound one may use the same method as the starting point.
This philosophy may allow us to make progress in other hybrid settings (see \cite{Huanguniform}).
However, technically speaking, to estimate those complicated sums is much more difficult. We have to take care of both aspects carefully.
It is worth mentioning that, as in Lin--Sun \cite{LinSun}, we drop the conductor-lowering trick which was used in Munshi \cite{Munshi2015circleIII} for the $t$-aspect, but we still use the conductor-lowering trick for the $M$-aspect as in Munshi \cite{Munshi2015circleIV} and Sharma \cite{Sharma2019}.


\subsection{Sketch of the proof} We give a brief sketch of the proof.
By the approximate functional equation we need to estimate
\[
  \sum_{n\sim N} A(r,n)\lambda_f(n) \chi(n) n^{-it}.
\]
We will apply the Duke--Friedlander--Iwaniec delta method with moduli $q\leq Q$ (see Lemma \ref{lemma:delta}).
For simplicity let us focus on the generic case, \emph{i.e.} $N = M^3 t^3$, $r = 1$ and $q \sim  Q = (LN/MK)^{1/2} $ 
for some parameters $L$ and $K\ll t^{1-\varepsilon}$ which will be chosen later.
After applying the DFI delta method and the conductor-lowering trick for the $M$-aspect by Munshi (see Sharma \cite{Sharma2019}), the main object of study is given by
\begin{multline*}
  \frac{1}{L}\sum_{\ell \sim L} \overline{A(1,\ell)} \int_{x\sim 1}
  \frac{1}{M}  \; \sideset{}{^\star}\sum_{b\bmod M}
  \frac{1}{Q}\sum_{\substack{q\sim Q \\ (q,\ell M)=1}}
  \frac{1}{q} \; \sideset{}{^\star}\sum_{a\bmod q}
  \sum_{n\sim LN} A(1,n)   e\left(\frac{n(aM+bq)}{qM} \right) e\left(\frac{n x}{MqQ}\right)  \\
   \cdot
  \sum_{m\sim N }\lambda_f(m) \chi(m)   e\left(\frac{-m\ell (aM+bq)}{qM}\right) m^{-it}
  e\left(\frac{-m\ell x}{MqQ}\right)\mathrm{d}x.
\end{multline*}
Trivially estimating at this stage gives $O(LN^2)$. So we want to save $LN$ plus a `little more' in the above sum. Note that here we don't need the conductor-lowing trick for the $t$-aspect as observed in \cite{Aggarwal,Huang,LinSun}. In fact, the $x$-integral above plays the same role as the $v$-integral in Munshi \cite{Munshi2015circleIII}.

We apply the Voronoi summation formulas to both $n$ and $m$ sums.
For the $n$ sum, by the $\GL(3)$ Voronoi, we get essentially
\[
  qM \sum_{n_2=1}^\infty \frac{A(1,n_2)}{n_2} S(\overline{(aM+bq)},n_2;qM) \Psi_x\left(\frac{n_2}{q^3M^3}\right),
\]
for certain weight function $\Psi_x$ depending on $x$.
The conductor is $K^3 M^3 Q^3$, and hence the dual length becomes $n_2\asymp \frac{K^3 M^3 Q^3}{LN}= \frac{L^2N^2}{Q^3}$.  By Lemma \ref{Huang lem 4.1}, the trivial bound for this dual sum is
$QM\cdot (QM)^{1/2}\cdot (\frac{LN}{MQ^2})^{3/2}$. So
we save $(LN)^{1/4}/(MK)^{3/4}$.
In the $\GL(2)$ Voronoi, the dual sum becomes essentially
\[
  \frac{N}{Mq\tau(\bar{\chi})}\sum_{\substack{u \bmod M \\u\not\equiv b \bmod M}}\bar{\chi}(u\ell) \sum_{m\geq 1}
  \lambda_f(m)e\left(\pm\frac{m\overline{\ell(aM+(b-u)q)}}{Mq}\right)
  H^{\pm}\left(\frac{mN}{M^2q^2}\right)
\]
for certain weight function $H^\pm$.
The conductor is $t^2 Q^2 M^2$, so  the dual length becomes $m\asymp \frac{t^2 Q^2 M^2}{N}= LM t^2 /K$.
By Lemma \ref{lem:VSF2} and the square root cancellation in the $u$ sum, the trivial bound
for this dual sum is $\frac{N}{Q M}\cdot \frac{M^{1/2}Q^{1/2}}{N^{1/4}} \cdot (\frac{t^2 Q^2 M^2}{N})^{3/4}\cdot \frac{1}{t^{1/2}}$.
Hence we save
$ N^{1/2}K^{1/2}/(L^{1/2}M^{1/2} t )$.
By the stationary phase method, we save $K^{1/2}$ from the $x$-integral.
We also save $Q^{1/2}$ in the $a$ sum and $M^{1/2}$ in the $b$ sum.
Hence in total we have saved
\[
  \frac{(LN)^{1/4}}{(MK)^{3/4}} \cdot \frac{N^{1/2}K^{1/2}}{L^{1/2}M^{1/2} t} \cdot K^{1/2} Q^{1/2} M^{1/2}
  = \frac{N}{Mt}.
\]
Generally we arrive at
\[
  \frac{N^{13/12}}{M^2 LQ} \sum_{\ell\in \mathcal{L}} \overline{A(1,\ell)\chi(\ell)} \ell^{1/3}
  \sum_{q\sim Q} \frac{1}{q^{3/2}} \sum_{n_2 \asymp \frac{L^2N^2}{Q^3}}
  \frac{A(1,n_2)}{n_2^{2/3}} \sum_{m\asymp \frac{M^2Q^2t^2}{N}} \frac{\lambda_f(m)}{m^{1/4}} \mathcal{C} \mathcal{J},
\]
for certain character sum $\mathcal{C}$ and integral transform $\mathcal{J}$ (see \eqref{eqn: re goal of case a}).

Next applying the Cauchy inequality we arrive at
\[
  \bigg( \sum_{n_2 \asymp \frac{L^2N^2}{Q^3}}
  \bigg|
  \sum_{\ell\in \mathcal{L}} \overline{A(1,\ell)\chi(\ell)} \ell^{1/3}
  \sum_{q\sim Q} \frac{1}{q^{3/2}}  \sum_{m\asymp \frac{M^2Q^2t^2}{N}} \frac{\lambda_f(m)}{m^{1/4}} \mathcal{C} \mathcal{J} \bigg|^2 \bigg)^{1/2},
\]
where we seek to save $L M t$ plus extra.
Opening the absolute value square we apply the Poisson summation formula on the sum over $n_2$. For the zero frequency we save $(LQ\frac{M^2Q^2 t^2}{N})^{1/2}$.
This gives a bound of size $\frac{N^{3/4}M^{3/4}K^{3/4}}{L^{1/4}}$.
We save enough in the zero frequency  if
$K<t$ and $L > 1$.

For the non-zero frequencies, the conductor is of size $Q^2 M K$, hence the length of the dual sum is $O((\frac{Q^2 M K}{L^2N^2/Q^3})^{1/2}) =O(\frac{L^{1/4}N^{1/4}}{M^{3/4}K^{3/4}})$.
In the integral transform we save $K^{1/4}$ and the character sums save $(Q^2M^{1/2})^{1/2}=QM^{1/4}$. Hence in total in the non-zero frequencies we save
$\frac{M^{3/4}K^{3/4}}{L^{1/4}N^{1/4}} K^{1/4} Q M^{1/4}$. This gives a bound of size $N^{1/4}Q L^{1/4} Mt = N^{3/4} L^{3/4}M^{1/2} \frac{t}{K^{1/2}}$. We save enough in the non-zero frequencies if $L<M^{1/3}$ and $K>t^{1/2}$. We also have different bounds from other cases. In fact, the best choice is $L=M^{1/4}$ and $K=t^{4/5}$ which gives $O(N^{1/2+\varepsilon} M^{3/2-1/16}t^{3/2-3/20})$ as claimed.

\subsection{Plan for this paper}
The rest of this paper is organized as follows.
In \S \ref{sec:preliminaries}, we introduce some notation and present some lemmas that we will need later.
The approximate functional equation allows us to reduce the subconvexity problem to estimating certain convolution sums.
In \S \ref{sec:reduction}, we apply the delta method to the convolution sums.
In \S \ref{sec:voronoi}, we apply the Voronoi summation formulas and estimate the integral transforms by the stationary phase method.
In \S \ref{sec:cauchy+poisson}, we apply the Cauchy--Schwarz inequality and Poisson summation formula, and then analyse the integrals.
Then we deal with character sums and the zero frequency contribution in \S \ref{sec:zero-freq}.
In \S \ref{sec:non-zero-freq}, we give the contribution from non zero frequencies.
Finally, in \S \ref{section:proof of Proposition 3.1}, we balance parameters optimally and prove
Proposition \ref{reduction prop} which leads to Theorem \ref{thm:subconv3*2}.

\medskip
\textbf{Notation.}
Throughout the paper, $\varepsilon$ is an arbitrarily small positive number;
all of them may be different at each occurrence.
By a smooth dyadic subdivision of a sum $\sum_{n\geq 1}A(n)$,
we will mean
$$
\sum_{(V,N)}\sum_{n\geq 1}A(n)V\left(\frac{n}{N}\right),
$$
where
$$
\sum_{(V,N)}V\left(\frac{n}{N}\right)=1
$$
with $V$ being a smooth function supported on $[1,2]$
and satisfying $V^{(j)}(x)\ll_j1$.
The weight functions $U,\ V,\ W$ may also change at each occurrence.
As usual, $e(x)=e^{2\pi i x}$.


\section{Preliminaries}\label{sec:preliminaries}

\subsection{Automorphic forms}

Let  $f$ be a Hecke--Maass cusp form with the spectral parameter $t_f$ for $\SL(2,\mathbb{Z})$, with the normalized Fourier coefficients $\lambda_f(n)$.
Let $\theta_2$ be the bound toward to the Ramanujan conjecture and we have $\theta_2\leq 7/64$ due to Kim--Sarnak \cite{Kim2003}.
It is well known that, by the Rankin--Selberg theory, one has
\begin{equation}\label{eqn:RS2}
  \sum_{n\leq N} |\lambda_f(n)|^2 \ll_f N.
\end{equation}

Let $\pi$ be a Hecke--Maass cusp form of type $(\nu_1,\nu_2)$ for $\SL(3,\mathbb{Z})$
with the normalized Fourier coefficients $A(r,n)$.
Similarly, Rankin--Selberg theory gives
\begin{equation}\label{eqn:RS3}
  \sum_{r^2n\leq N} |A(r,n)|^2 \ll_\pi   N.
\end{equation}

We record the Hecke relation
\[
  A(r,n) = \sum_{d\mid (r,n)} \mu(d) A\left(\frac{r}{d},1\right) A\left(1,\frac{n}{d}\right)
\]
which follows from M\"obius inversion and \cite[Theorem 6.4.11]{goldfeld2006automorphic}.
Hence we have the individual bounds
\begin{equation}\label{eqn:Bound A(r,n)}
  A(r,n) \ll (rn)^{\theta_3+\varepsilon},
\end{equation}
where $\theta_3\leq 5/14$ is the bound toward to the Ramanujan conjecture on $\GL(3)$ (see \cite{Kim2003}).
So we have
\begin{equation}\label{eqn:RS3-1}
  \sum_{n\sim N} |A(r,n)| \ll \sum_{n_1\mid r^\infty} \sum_{\substack{n\sim N/n_1 \\ (n,r)=1}} |A(r,nn_1)| \leq \sum_{n_1\mid r^\infty}|A(r,n_1)| \sum_{\substack{n\sim N/n_1 \\ (n,r)=1}} |A(1,n)| \ll r^{\theta_3+\varepsilon} N
\end{equation}
and
\begin{equation}\label{eqn:RS3-2}
  \sum_{n\sim N} |A(r,n)|^2
  \ll \sum_{n_1\mid r^\infty} \sum_{\substack{n\sim N/n_1 \\ (n,r)=1}} |A(r,nn_1)|^2
  \leq \sum_{n_1\mid r^\infty}|A(r,n_1)|^2 \sum_{\substack{n\sim N/n_1 \\ (n,r)=1}} |A(1,n)|^2
  \ll r^{2\theta_3+\varepsilon} N.
\end{equation}
Here we have used \eqref{eqn:RS3} and the fact
$
  \sum_{d\mid r^\infty} d^{-\sigma} \ll r^\varepsilon, \; \textrm{for $\sigma>0$}.
$

\subsection{$L$-functions}\label{subsec:L-functions}

The Rankin--Selberg $L$-function $L(s,\pi\times f \times \chi)$ has the following functional equation
\[
  \Lambda(s,\pi\times f \times \chi) = \epsilon_{\pi\times f\times \chi} \Lambda(1-s,\tilde\pi\times f \times \bar\chi),
\]
where
\[
  \Lambda(s,\pi\times f \times \chi) = M^{3s} \pi^{-3s} \prod_{j=1}^{3} \prod_{\pm}
  \Gamma\left(\frac{s-\alpha_j\pm i t_f}{2}\right) L(s,\pi\times f \times \chi)
\]
is the completed $L$-function and $\epsilon_{\pi\times f\times \chi}$ is the root number.
Here $\alpha_j$ are the Langlands parameters of $\pi$, and $\tilde\pi$ is the contragredient representation of $\pi$.
By \cite[\S5.2]{IwaniecKowalski2004analytic},
we can obtain the approximate functional equation which leads us to the following result.

\begin{lemma}\label{lemma:AFE}
  We have
  \[
    L(1/2+it,\pi\times f \times \chi) \ll (M(|t|+1))^\varepsilon \sup_{N \ll (M(|t|+1))^{3+\varepsilon}} \frac{|S(N)|}{\sqrt{N}} + (M(|t|+1))^{-A},
  \]
  where
  \[
    S(N) = \sum_{r\geq1} \sum_{n\geq1} A(r,n)\lambda_f(n) \chi(r^2 n) (r^2n)^{-it} V\left(\frac{r^2n}{N}\right),
  \]
  with some compactly supported smooth function $V$ such that $\supp V\subset [1,2]$ and  $V^{(j)}\ll_j 1$.
\end{lemma}


We first estimate the contribution from large values of $r$. By \eqref{eqn:RS2} and \eqref{eqn:RS3-2} we have
\begin{align}\label{eqn:r-large}
  \sum_{r\geq M^{1/8} (|t|+1)^{3/10}} & \left|\sum_{n\geq1} A(r,n)\lambda_f(n) \chi(n) (r^2n)^{-it} V\left(\frac{r^2n}{N}\right)\right| \nonumber \\
  & \ll \sum_{ M^{1/8} (|t|+1)^{3/10} \leq r \ll \sqrt{N}} \left(\sum_{n\asymp N/r^2} |A(r,n)|^2 \right)^{1/2} \left(\sum_{n\asymp N/r^2}  |\lambda_f(n)|^2\right)^{1/2}
  \nonumber \\
  & \ll \sum_{ M^{1/8} (|t|+1)^{3/10} \leq r \ll \sqrt{N}} r^{\theta_3+\varepsilon} \frac{N}{r^2}
  \ll N \sum_{ M^{1/8} (|t|+1)^{3/10} \leq r \ll \sqrt{N}} r^{-3/2-\varepsilon} \nonumber \\
  &
  \ll N^{1/2} M^{3/2-1/16} (|t|+1)^{3/2-3/20+\varepsilon},
\end{align}
for $N \ll (M(|t|+1))^{3+\varepsilon}$.
The contribution from those terms to $L(1/2+it,\pi\times f \times \chi)$ is bounded by $M^{3/2-1/16+\varepsilon} (|t|+1)^{3/2-3/20+\varepsilon}$.

Therefore, combining this together with Lemma \ref{lemma:AFE}, we prove the following lemma.

\begin{lemma}\label{lemma:L<<SrN}
We have
\begin{equation*}
    L(1/2+it,\pi\times f \times \chi) \ll t^\varepsilon
    \sum_{\substack{r\leq M^{1/8} t^{3/10} \\ (r,M)=1}} \frac{1}{r}  \sup_{N \ll (Mt)^{3+\varepsilon}/r^2} \frac{|S(r,N)|}{\sqrt{N}}
    + M^{3/2-1/16} t^{3/2-3/20+\varepsilon},
\end{equation*}
where
\[
  S(r,N) := \sum_{n\geq1} A(r,n)\lambda_f(n) \chi(n) n^{-it} V\left(\frac{n}{N}\right).
\]
\end{lemma}


\subsection{Summation formulas}

We first recall
the Poisson summation formula over an arithmetic progression.
\begin{lemma}\label{lem:Poisson}
  Let $\beta\in\mathbb{Z}$ and $c\in \mathbb{Z}_{\geq1}$. For a Schwartz function $f:\mathbb{R}\rightarrow \mathbb{C}$, we have
  \[
    \sum_{\substack{n\in\mathbb{Z}\\ n\equiv \beta \bmod{c}}} f(n) = \frac{1}{c} \sum_{n\in\mathbb{Z}} \hat{f}\left(\frac{n}{c}\right) e\left(\frac{n\beta}{c}\right),
  \]
  where $\hat{f}(y)=\int_{\mathbb{R}} f(x) e(-xy)\dd x$ is the Fourier transform of $f$.
\end{lemma}

\begin{proof}
  See e.g. \cite[Eq. (4.24)]{IwaniecKowalski2004analytic}.
\end{proof}

We recall the Voronoi summation formula for $\SL(2,\mathbb{Z})$.
Let $g$ be a smooth compactly supported function on $(0,\infty)$.
\begin{lemma}\label{lem:VSF2}
  With the notation as above.  Then we have
  \begin{equation}\label{eqn:VSF2}
    \sum_{n\geq1} \lambda_f(n) e\left(\frac{an}{q}\right) g\left(\frac{n}{N}\right)
    = \frac{N}{q} \sum_{\pm} \sum_{n\geq1} \lambda_f(n) e\left(\mp\frac{\bar{a} n}{q}\right) H^\pm \left(\frac{nN}{q^2}\right)
  \end{equation}
  where
  \begin{equation}\label{eqn:G+}
    H^+ (y) =   \frac{-\pi}{\sin(\pi i t_f)} \int_{0}^{\infty} g(\xi) (J_{2it_f}(4\pi\sqrt{y\xi})-J_{-2it_f}(4\pi\sqrt{y\xi})) \dd \xi,
  \end{equation}
  and
  \begin{equation}\label{eqn:G-}
    H^- (y) =  4 \epsilon_f \cosh(\pi t_f) \int_{0}^{\infty} g(\xi) K_{2it_f}(4\pi\sqrt{y\xi}) \dd \xi.
  \end{equation}
  For $y\gg T^\varepsilon$, we have
  \begin{align}\label{J estimates}
    H^+ (y) = y^{-1/4} \int_{0}^{\infty} g(\xi) \xi^{-1/4} \sum_{j=0}^{J} \frac{c_j e(2\sqrt{y\xi})+ d_j e(-2\sqrt{y\xi})}{(y\xi)^{j/2}} \dd \xi + O(T^{-A})
  \end{align}
  for some constant $J=J(A)$
  and
  \begin{align}\label{K estimates}
    H^- (y) \ll_{t_f,A} y^{-A}.
  \end{align}
\end{lemma}

\begin{proof}
  See e.g. \cite[\S3.1]{LinSun}.
\end{proof}
Notice that \eqref{J estimates} and \eqref{K estimates} are only valid for $y\gg T^\varepsilon$.
So we also need the facts which state that, for $y>0$, $k\geq 0$ and $\Re\nu=0$, one has (see \cite[Lemma C.2]{KMV2002})
\begin{align}\label{Bessel for small y}
  \begin{split}
     &y^kJ_{\nu}^{(k)}(y)\ll_{k,\nu}\frac{1}{(1+y)^{1/2}},
      \\
     &y^kK_{\nu}^{(k)}(y)\ll_{k,\nu}\frac{e^{-y}(1+|\log y|)}{(1+y)^{1/2}}.
  \end{split}
\end{align}

We now recall the Voronoi summation formula for $\SL(3,\mathbb{Z})$.
Let $\psi$ be a smooth compactly supported function on $(0,\infty)$,
and let $\tilde{\psi}(s):=\int_{0}^{\infty}\psi(x)x^s\frac{\dd x}{x}$
be its Mellin transform.
For $\sigma>5/14$, we define
\begin{equation}\label{eqn:Psi}
    \Psi^{\pm}(z) := z \frac{1}{2\pi i} \int_{(\sigma)} (\pi^3z)^{-s} \gamma_3^\pm(s) \tilde{\psi}(1-s)\dd s ,
\end{equation}
with
\begin{equation}\label{eqn:gamma^pm}
  \gamma_3^\pm(s) := \prod_{j=1}^{3}
    \frac{\Gamma\left(\frac{s+\alpha_j}{2}\right)} {\Gamma\left(\frac{1-s-\alpha_j}{2}\right)}
    \pm \frac{1}{i} \prod_{j=1}^{3}
    \frac{\Gamma\left(\frac{1+s+\alpha_j}{2}\right)} {\Gamma\left(\frac{2-s-\alpha_j}{2}\right)},
\end{equation}
where $\alpha_j$ are the Langlands parameters of $\pi$ as above.
Note that changing $\psi(y)$ to $\psi(y/N)$ for a positive real number $N$ has the effect of
changing $\Psi^\pm(z)$ to $\Psi^\pm(zN)$.
The Voronoi formula on $\GL(3)$ was first proved by Miller--Schmid~\cite{MillerSchmid2006automorphic}.
The present version is due to Goldfeld--Li~\cite{goldfeld2006voronoi} with slightly renormalized variables (see Blomer \cite[Lemma 3]{blomer2012subconvexity}).
\begin{lemma}\label{lemma:VSF3}
  Let $c,d,\bar{d}\in\mathbb Z$ with $c\neq0$, $(c,d)=1$, and $d\bar{d}\equiv1\pmod{c}$.
  Then we have
  \begin{equation*}
    \begin{split}
         \sum_{n=1}^{\infty} A(m,n)e\left(\frac{n\bar{d}}{c}\right)\psi(n)
         = \frac{c\pi^{3/2}}{2} \sum_{\pm} \sum_{n_1|cm} \sum_{n_2=1}^{\infty}
              \frac{A(n_2,n_1)}{n_1n_2} S\left(md,\pm n_2;\frac{mc}{n_1}\right)
              \Psi^{\pm}\left(\frac{n_1^2n_2}{c^3m}\right),
    \end{split}
  \end{equation*}
  where $S(a,b;c) := \mathop{{\sum}^*}_{d(c)} e\left(\frac{ad+b\bar{d}}{c}\right)$ is the classical Kloosterman sum.
\end{lemma}

%
%
%



\subsection{The delta method}

There are two oscillatory factors contributing to the convolution sums. Our method is based on separating these oscillations using the circle method. In the present situation we will use a version of the delta method of Duke, Friedlander and Iwaniec. More specifically we will use the expansion (20.157) given in  \cite[\S20.5]{IwaniecKowalski2004analytic}. Let $\delta:\mathbb{Z}\rightarrow \{0,1\}$ be defined by
$$
\delta(n)=\begin{cases} 1&\text{if}\;\;n=0;\\
0&\text{otherwise}.\end{cases}
$$
We seek a Fourier expansion which matches with $\delta(n)$.
\begin{lemma}\label{lemma:delta}
  Let $Q$ be a large positive number. Then we have
  \begin{align}\label{eqn:delta-n}
    \delta(n)=\frac{1}{Q}\sum_{1\leq q\leq Q} \;\frac{1}{q}\; \sideset{}{^\star}\sum_{a\bmod{q}}e\left(\frac{na}{q}\right)
    \int_\mathbb{R} g(q,x) e\left(\frac{nx}{qQ}\right)\mathrm{d}x,
  \end{align}
  where 
  $g(q,x)$ is a weight function satisfying that
  \begin{equation}\label{eqn:g-h}
    g(q,x)=1+O\left(\frac{Q}{q}\left(\frac{q}{Q}+|x|\right)^A\right),
    \quad
     g(q,x)\ll |x|^{-A}, \quad \textrm{for any $A>1$},
  \end{equation}
  and
  \begin{equation}\label{eqn:g^(j)}
    \frac{\partial^j}{\partial x^j} g(q,x) \ll  |x|^{-j} \min(|x|^{-1},Q/q) \log Q, \quad j\geq1.
  \end{equation}
  Here the $\star$ on the sum indicates that the sum over $a$ is restricted by the condition $(a,q)=1$.
\end{lemma}

\begin{proof}
  See \cite[Lemma 15]{Huang} and \cite[\S20.5]{IwaniecKowalski2004analytic}.
\end{proof}

In applications of \eqref{eqn:delta-n}, we can first restrict to $|x|\ll Q^\varepsilon$. If $q\gg Q^{1-\varepsilon}$, then by \eqref{eqn:g^(j)} we get $ \frac{\partial^j}{\partial x^j} g(q,x) \ll Q^\varepsilon |x|^{-j} $, for any $j\geq1$. If $q\ll Q^{1-\varepsilon}$ and $Q^{-\varepsilon} \ll |x| \ll Q^\varepsilon$, then by \eqref{eqn:g^(j)} we also have $ \frac{\partial^j}{\partial x^j} g(q,x) \ll Q^\varepsilon |x|^{-j} $, for any $j\geq1$. Finally, if $q\ll Q^{1-\varepsilon}$ and $|x| \ll Q^{-\varepsilon}$, then by \eqref{eqn:g-h}, we can replace $g(q,x)$ by 1 with a negligible error term.
So in all cases, we can view $g(q,x)$ as a nice weight function.

We remark that
there is no restrictions on $Q$,
so we can choose $Q$ to be any large positive number.
Recall that in Sharma \cite{Sharma2019} and Lin--Sun \cite{LinSun},
the authors took $Q$ to be $(\frac{NL}{M})^{1/2}$ and $(\frac{N}{t^{4/5}})^{1/2}$,
respectively. This motivates us to choose $Q=(\frac{NL}{MK})^{1/2}$.
As we will see, after balancing finally, we can take
$L=M^{1/4}$ and $K=t^{4/5}$ optimally, which coincides with Sharma \cite{Sharma2019} and Lin--Sun \cite{LinSun}.

\subsection{Oscillatory integrals}

Let $\mathcal{F}$ be an index set and $X=X_T:\mathcal{F}\rightarrow \mathbb{R}_{\geq1}$ be a function of $T\in\mathcal{F}$.
A family of $\{w_T\}_{T\in\mathcal{F}}$ of smooth functions supported on a product of dyadic intervals in $\mathbb{R}_{>0}^d$ is called $X$-inert if for each $j=(j_1,\ldots,j_d) \in \mathbb{Z}_{\geq0}^d$ we have
\[
  \sup_{T\in\mathcal{F}} \sup_{(x_1,\ldots,x_d) \in \mathbb{R}_{>0}^d}
  X_T^{-j_1-\cdots -j_d} \left| x_1^{j_1} \cdots x_d^{j_d} w_T^{(j_1,\ldots,j_d)} (x_1,\ldots,x_d) \right|
   \ll_{j_1,\ldots,j_d} 1.
\]
We will use the following stationary phase lemma several times.

\begin{lemma}\label{lemma:stationary_phase}
  Suppose $w=w_T(t)$ is a family of $X$-inert functions, with compact support on $[Z,2Z]$,
  so that $w^{(j)}(t)\ll (\frac{Z}{X})^{-j}$. Also suppose that $\phi$ is smooth and satisfies
  $\phi^{(j)}\ll \frac{Y}{Z^j}$ for some $\frac{Y}{X^2}\geq R\geq 1$ and all $t$ in the support
  of $w$. Let
  \[
    I=\int_{-\infty}^\infty w(t)e^{i\phi(t)}\textup{d}t.
  \]
  \begin{enumerate}
  \item [(i)] If $|\phi'(t)|\gg \frac{Y}{Z}$ for all $t$ in the support of $w$, then $I\ll_A ZR^{-A}$
  for $A$ arbitrarily large.
  \item [(ii)] If $|\phi''(t)|\gg \frac{Y}{Z^2}$ for all $t$ in the support of $w$, and there exists
  $t_0\in \mathbb{R}$ such that $\phi'(t_0)=0$ (note that $t_0$ is necessarily unique),
  then
  \[
    I
    = \frac{e^{i\phi(t_0)}}{\sqrt{\phi''(t_0)}}F_T(t_0) + O_A(ZR^{-A}),
  \]
  where $F_T$ is a family of $X$-inert functions (depending on $A$) supported on $t_0\asymp Z$.
  \end{enumerate}
\end{lemma}
\begin{proof}
  See \cite[\S 8]{BlomerKhanYoung} and \cite[Lemma 3.1]{KPY}.
\end{proof}


\section{Reduction}\label{sec:reduction}

Now we start to prove Theorem \ref{thm:subconv3*2}. We assume $t\geq M^{\varepsilon}$.
Recall that, Lemma  \ref{lemma:L<<SrN}, we are considering $S(r,N)$ with $N\ll (Mt)^{3+\varepsilon}/r^{2}$, $r\ll M^{1/8}t^{3/10}$, and $(r,M)=1$.
We will prove the following proposition.


\begin{proposition}\label{reduction prop}
  We have
  \[
    S(r,N) \ll N^{1/2+\varepsilon} M^{3/2-1/16} t^{3/2-3/20} ,
  \]
  for $N\ll (Mt)^{3+\varepsilon}/r^2$, $r\ll M^{1/8} t^{3/10}$ and $(r,M)=1$.
\end{proposition}

Let $\mathcal{L}$ be the set of primes in $[L,2L]$.
Assume $ M\notin [L,2L] $.  For $\ell\in\mathcal{L}$ and $n\geq1$, by the Hecke relation, we have
\[
  A(1,\ell) A(r,n) = A(r,\ell n) + \delta_{\ell|r}A(r/\ell,n) + \delta_{\ell|n}A(r\ell,n/\ell).
\]
By the prime number theorem for $L(s,\pi\times \tilde \pi)$ we have
\[
  L^* := \sum_{\ell\in\mathcal{L}} |A(1,\ell)|^2 \gg L^{1-\varepsilon}.
\]
We have
\begin{align*}
   S(r,N) = & \frac{1}{L^*} \sum_{\ell\in\mathcal{L}} \overline{A(1,\ell)} \sum_{n\geq1} A(r,n)A(1,\ell) \lambda_f(n)\chi(n)  n^{-it} V\left(\frac{n}{N}\right)
   \\
  = & S_1(N)+S_2(N)+S_3(N),
\end{align*}
where
\[
  S_1(N) = \frac{1}{L^*} \sum_{\ell\in\mathcal{L}} \overline{A(1,\ell)} \sum_{n\geq1} A(r,n\ell)\lambda_f(n) \chi(n)  n^{-it} V\left(\frac{n}{N}\right),
\]
\begin{align*}
  S_2(N) & = \frac{1}{L^*} \sum_{\ell\in\mathcal{L}} \overline{A(1,\ell)} \sum_{n\geq1} \delta_{\ell|r}A(r/\ell,n) \lambda_f(n) \chi(n)  n^{-it} V\left(\frac{n}{N}\right),
\end{align*}
and
\begin{align*}
  S_3(N) & =\frac{1}{L^*} \sum_{\ell\in\mathcal{L}} \overline{A(1,\ell)} \sum_{n\geq1} \delta_{\ell|n}A(r\ell,n/\ell) \lambda_f(n) \chi(n)  n^{-it} V\left(\frac{n}{N}\right).
\end{align*}

We only consider $S_1(N)$, since the same method works for the other two sums and will give better bounds as the lengths of those sums are smaller.
Actually, in $S_2$, since $\ell \mid r$, only $\tau(r)$ $\ell$'s contribute;
in $S_3$, since $\ell \mid n$, the length of the $n$-sum is of size $\frac{N}{L}$.
As the structures of sums in $S_2$ and $S_3$ are the same as in $S_1$, we can get better bounds
than $S_1$.
Now we apply
$
  \frac{1}{M}\underset{{b(\bmod M)}}{\sum}e(\frac{(n-m\ell)b}{M})
$ to detect the condition $M\mid (n-m\ell)$, and then use the delta method, obtaining
\begin{align*}
  S_1(N)  & = \frac{1}{L^*} \sum_{\ell\in\mathcal{L}} \overline{A(1,\ell)} \frac{1}{M} \sum_{b\bmod M}
  \sum_{n\geq1} A(r,n) W\left(\frac{n}{\ell N}\right) \\
  & \hskip 50pt \cdot
  \sum_{m\geq1 } \lambda_f(m)\chi(m)  m^{-it} V\left(\frac{m}{N}\right) e\left(\frac{(n-m\ell)b}{M}\right) \\
  & \hskip 50pt \cdot
  \frac{1}{Q}\sum_{1\leq q\leq Q} \;\frac{1}{q}\; \sideset{}{^\star}\sum_{a\bmod{q}}e\left(\frac{(n-m\ell)a}{Mq}\right) \int_\mathbb{R}g(q,x) e\left(\frac{(n-m\ell)x}{MqQ}\right)\mathrm{d}x .
\end{align*}
Rearranging the order of the sums and integrals we get
\begin{align*}
  S_1(N)
  & = \frac{1}{L^*} \sum_{\ell\in\mathcal{L}} \overline{A(1,\ell)} \frac{1}{M}\sum_{b\bmod M}
  \frac{1}{Q}\sum_{1\leq q\leq Q} \int_\mathbb{R}g(q,x) \;\frac{1}{q}\; \sideset{}{^\star}\sum_{a\bmod{q}} \\
  & \hskip 50pt \cdot
  \sum_{n\geq1} A(r,n) e\left(\frac{n(bq+a)}{Mq}\right) W\left(\frac{n}{\ell N}\right)
   e\left(\frac{nx}{pqQ}\right)
   \\
  & \hskip 50pt \cdot \sum_{m\geq1 } \lambda_f(m)\chi(m)  m^{-it} e\left(\frac{-m\ell (bq+a)}{Mq}\right)  V\left(\frac{m}{N}\right)
  e\left(\frac{-m\ell x}{MqQ}\right)  \mathrm{d}x.
\end{align*}
Inserting a smooth partition of unity for the $x$-integral and a dyadic partition for the $q$-sum, we get
\[
  S_1(N) \ll  N^\varepsilon \sup_{t^{-B} \ll X \ll t^\varepsilon}
  \sup_{1\ll R\ll Q} \sum_{1\leq j\leq 3} |S_{1j}^{\pm}(N,X,R)|  + O(t^{-A}),
\]
for any large constant $A>0$ and some large constant $B>0$ depending on $A$,
where
$S_{11}^{\pm}(N,X,R)$, $S_{12}^{\pm}(N,X,R)$ and $S_{13}^{\pm}(N,X,R)$
denote the terms with $(b,M)=1$, $M \mid b$ and $(q,\ell M)>1$, respectively.
More precisely, we have
\begin{align}\label{S11 pm}
 \begin{split}
  S_{11}^{\pm}(N,X,R)  & = \frac{1}{L^*} \sum_{\ell\in\mathcal{L}} \overline{A(1,\ell)} \int_\mathbb{R}
  \frac{1}{M}  \; \sideset{}{^\star}\sum_{b\bmod M}
  \frac{1}{Q}\sum_{\substack{q\sim R \\ (q,\ell M)=1}}
  \frac{1}{q} \; \sideset{}{^\star}\sum_{a\bmod q} g(q,x)U\left(\frac{\pm x}{X}\right) \\
  & \hskip 15pt \cdot
  \sum_{n\geq1} A(r,n)   e\left(\frac{n(aM+bq)}{qM} \right) W\left(\frac{n}{\ell N}\right) e\left(\frac{n x}{MqQ}\right)  \\
  & \hskip 15pt \cdot
  \sum_{m\geq1 }\lambda_f(m) \chi(m)   e\left(\frac{-m\ell (aM+bq)}{qM}\right) m^{-it} V\left(\frac{m}{N}\right)
  e\left(\frac{-m\ell x}{MqQ}\right)\mathrm{d}x,
 \end{split}
\end{align}
\begin{align*}
  S_{12}^{\pm}(N,X,R)  & = \frac{1}{L^*} \sum_{\ell\in\mathcal{L}} \overline{A(1,\ell)} \int_\mathbb{R}
  \frac{1}{M}
  \frac{1}{Q}\sum_{\substack{q\sim R \\ (q,\ell M)=1}}
  \frac{1}{q} \; \sideset{}{^\star}\sum_{a\bmod q} g(q,x)U\left(\frac{\pm x}{X}\right) \\
  & \hskip 15pt \cdot
  \sum_{n\geq1} A(r,n)   e\left(\frac{na}{q} \right) W\left(\frac{n}{\ell N}\right) e\left(\frac{n x}{MqQ}\right)  \\
  & \hskip 15pt \cdot
  \sum_{m\geq1 }\lambda_f(m) \chi(m)   e\left(\frac{-m\ell a}{q}\right) m^{-it} V\left(\frac{m}{N}\right)
  e\left(\frac{-m\ell x}{MqQ}\right)\mathrm{d}x,
\end{align*}
and
\begin{align*}
  S_{13}^{\pm}(N,X,R)  & = \frac{1}{L^*} \sum_{\ell\in\mathcal{L}} \overline{A(1,\ell)} \int_\mathbb{R}
  \frac{1}{M} \; \sum_{b\bmod M}
  \frac{1}{Q}\sum_{\substack{q\sim R \\ (q,\ell M)>1}}
  \frac{1}{q} \; \sideset{}{^\star}\sum_{a\bmod q} g(q,x)U\left(\frac{\pm x}{X}\right) \\
  & \hskip 15pt \cdot
  \sum_{n\geq1} A(r,n)   e\left(\frac{n(a+bq)}{qM} \right) W\left(\frac{n}{\ell N}\right) e\left(\frac{n x}{MqQ}\right)  \\
  & \hskip 15pt \cdot
  \sum_{m\geq1 }\lambda_f(m) \chi(m)   e\left(\frac{-m\ell (a+bq)}{qM}\right) m^{-it} V\left(\frac{m}{N}\right)
  e\left(\frac{-m\ell x}{MqQ}\right)\mathrm{d}x.
\end{align*}
Noth that in $S_{11}^{\pm}(N,X,R)$ and $S_{12}^{\pm}(N,X,R)$,
we have made a change of variable $a\rightarrow aM$.
Here $U$ is a fixed compactly supported $1$-inert function with  $\supp U \subset (0,\infty)$.
We will only give details for the treatment of $S_{11}^{\pm}(N,X,R)$, since the same method works for $S_{12}^{\pm}(N,X,R)$ and $S_{13}^{\pm}(N,X,R)$  and will give a better upper bound.
More precisely, in $S_{12}^{\pm}(N,X,R)$, we do not have the $b$-sum.
In $S_{13}^{\pm}(N,X,R)$, we have the condition $(q,\ell M)>1$.
In fact, we should have the following cases:
\begin{itemize}
\item [(i)] $b\equiv 0\bmod M$ and $q=\ell^j q'$ with $j\geq1$ and $(q',\ell M)=1$;
\item [(ii)] $b\equiv 0\bmod M$ and $q=M^k q'$ with $k\geq1$ and $(q',\ell M)=1$;
\item [(iii)] $b\equiv 0\bmod M$ and $q=\ell^j M^k q'$ with $j,k\geq1$ and $(q',\ell M)=1$;
\item [(iv)] $(b,M)=1$ and $q=\ell^j q'$ with $j\geq1$ and $(q',\ell M)=1$;
\item [(v)] $(b,M)=1$ and $q=M^k q'$ with $k\geq1$ and $(q',\ell M)=1$;
\item [(vi)] $(b,M)=1$ and $q=\ell^j M^k q'$ with $j,k\geq1$ and $(q',\ell M)=1$.
\end{itemize}


\section{Applying Voronoi}\label{sec:voronoi}
We first apply the Voronoi summation formula (see Lemma \ref{lemma:VSF3}) to the sum over $n$ in $S_{11}^{\pm}(N,X,R)$ , getting
\begin{multline}\label{eqn: use voronoi gl3}
  \sum_{n\geq1} A(r,n) e\left(\frac{n(bq+aM)}{qM}\right) W\left(\frac{n}{\ell N}\right)
   e\left(\frac{n x}{MqQ}\right) \\
   = qM \sum_{\eta_1=\pm 1}\sum_{n_1|qMr}\sum_{n_2=1}^\infty \frac{A(n_1,n_2)}{n_1n_2} S(r\overline{(aM+bq)},\eta_1n_2;qMr/n_1)\Psi_x^{\sgn(\eta_1)}\left(\frac{n_1^2n_2}{q^3M^3r}\right),
\end{multline}
where $\Psi_x^{\sgn(\eta_1)}(z)$ is defined as in Lemma \ref{lemma:VSF3} with $\psi(y)$ replaced by
$W(\frac{y}{\ell N})e(\frac{xy}{MqQ})$.
\begin{lemma}\label{Huang lem 4.1}
  \begin{itemize}
    \item [(i)] If $zNL\gg t^{\varepsilon}$, then $\Psi_x^{\eta_1}(z)$ is negligibly small
    unless $\sgn(x)=-\sgn(\eta_1)$ and $\frac{N\ell (-\eta_1x)}{MqQ}\asymp (zN\ell)^{1/3}$,
    in which case we have
    \begin{align*}
      \Psi_x^{\sgn(\eta_1)}(z)= (zN\ell)^{1/2}e\left(\eta_1 \frac{2(zMqQ)^{1/2}}{(-\eta_1x)^{1/2}}\right)
      \mathcal{W}\left(\frac{(z^{1/2}(MqQ)^{3/2}}{N\ell(-\eta_1x)^{3/2}}\right)+O(t^{-A}),
    \end{align*}
    where $\mathcal{W}$ is a certain compactly supported $1$-inert function depending on $A$.
    \item [(ii)] If $zNL\ll t^\varepsilon$ and $\frac{NL X}{MRQ}\gg t^\varepsilon$,
    then $\Psi_x^{\sgn(\eta_1)}(z)\ll t^{-A}$.
    \item [(iii)] If $zNL \ll t^\varepsilon$ and $\frac{NL X}{MRQ}\ll t^\varepsilon$,
    then $\Psi_x^{\sgn(\eta_1)}(z)\ll t^{\varepsilon}$.
  \end{itemize}
\end{lemma}
\begin{proof}
  See \cite[5.3]{Huang}.
\end{proof}

In the last case, by taking $\sigma=1/2$ and making a change of variable, we get
\begin{align*}
  \Psi_x^\pm(z) & = (z\ell N)^{1/2} \frac{1}{2\pi^{5/2} } \int_{\mathbb{R}} (\pi^3 z \ell N)^{-i\tau} \gamma_3^\pm(1/2+i\tau) \int_{0}^{\infty} W\left(\xi\right) e\left(\frac{x\ell N\xi}{MqQ} \right)  \xi^{-1/2-i\tau} \dd \xi  \dd \tau.
\end{align*}
We can truncate $\tau$ at $\tau\ll t^\varepsilon$ with a negligibly small error by repeated integration by parts for the $\xi$-integral above. That is,  we have
\begin{equation}\label{eqn:Psi=*W}
  \Psi_x^\pm(z)=(z\ell N)^{1/2} W_{x,\ell}^\pm(z)+O(t^{-A}),
\end{equation}
where
\begin{align*}
  W_{x,\ell}^\pm(z) & =  \frac{1}{2\pi^{5/2} } \int_{|\tau|\leq t^\varepsilon} (\pi^3 z \ell N)^{-i\tau} \gamma_3^\pm(1/2+i\tau) \int_{0}^{\infty} W\left(\xi\right) e\left(\frac{x\ell N\xi}{MqQ} \right)  \xi^{-1/2-i\tau} \dd \xi  \dd \tau.
\end{align*}
The contribution from the error to $S_{11}^\pm(N,X,R)$ is also negligibly small.
Note that the function $W_{x,\ell}^\pm(z)$ satisfies that
\begin{equation}\label{eqn:W(z)}
  \frac{\partial^j}{\partial z^j} W_{x,\ell}^\pm(z) \ll_j t^\varepsilon z^{-j}.
\end{equation}

Now we consider the $m$-sum.
By
\[
  \chi(m)=\bar{\chi}(\ell)\chi(m\ell)=\frac{\bar{\chi}(\ell)}{\tau(\bar{\chi})}\sum_{u\bmod M}\bar{\chi}(u)
  e\left(\frac{um\ell}{M}\right),
\]
one has
\begin{multline*}
  \sum_{m\geq1} \lambda_f(m)\chi(m) m^{-it} e\left(\frac{-m\ell (bq+aM)}{Mq}\right)  V\left(\frac{m}{N}\right)
  e\left(\frac{-m\ell x}{MqQ}\right)
  \\
  =\frac{1}{\tau(\bar{\chi})}  \sum_{m\geq 1}\lambda_f(m)m^{-it}V\left(\frac{m}{N}\right)e\left(-\frac{m\ell x}{MqQ}\right)
  \\
  \cdot \left(\sum_{\substack{u \bmod M \\ u\not\equiv b \bmod M}}
  \bar{\chi}(u\ell)
  \left(e\left(-\frac{m\ell(aM+(b-u)q)}{Mq}\right)+e\left(-\frac{m\ell a}{q}\right)\right)\right)
  =: \Sigma_1+\Sigma_2,
\end{multline*}
say.
From now on, we only deal with the terms involving $\Sigma_1$,
since the treatment of
$\Sigma_2$ is similar and in fact simpler.
With the help of Lemma \ref{lem:VSF2}, we obtain
\begin{multline}\label{eqn: use voronoi gl2}
  \Sigma_1
  =\frac{N^{1-it}}{Mq\tau(\bar{\chi})}\sum_{\substack{u \bmod M \\u\not\equiv b \bmod M}}\bar{\chi}(u\ell)\sum_{\pm}\sum_{m\geq 1}
  \lambda_f(m)e\left(\pm\frac{m\overline{\ell(aM+(b-u)q)}}{Mq}\right)
  H^{\pm}\left(\frac{mN}{M^2q^2}\right),
\end{multline}
where $H^{\pm}$ is defined as in Lemma \ref{lem:VSF2} with $g(\xi)$ replaced by $V(\xi)\xi^{-it}e(-\frac{N\ell x\xi}{MqQ})$.
\begin{lemma}\label{GL2 Hpm}
    If $z\ll t^\varepsilon$, then $H^\pm(z)$ is negligible unless $t\asymp\frac{N\ell X}{MqQ}$ and $x<0$.
\end{lemma}
\begin{proof}
  If $z\ll t^{\varepsilon}$, then, in view of \eqref{eqn:G+} and \eqref{eqn:G-}, we may regard $H^{\pm}(z)$ as
\begin{align}\label{eqn:mathcal(I)}
  \mathcal{I}(z):=\int_{0}^\infty V(\xi)e\left(-\frac{t\log \xi}{2\pi}-\frac{N\ell x\xi}{MqQ}\right)J_f(z\xi) \dd \xi,
\end{align}
where $J_f(z)=\frac{-\pi}{\sin(\pi it_f)}(J_{2it_f}(4\pi\sqrt{z})-J_{-2it_f}(4\pi\sqrt{z}))$ or
$J_f(z)=4\epsilon_f\cosh(\pi t_f)K_{2it_f}(4\pi\sqrt{z})$.
Then, by partial integration together with \eqref{Bessel for small y},
$\mathcal{I}_1(z)$ is negligible unless
$x<0$ and $\frac{NLX}{MRQ}\asymp t$.
\end{proof}

If $\frac{mN}{M^2q^2}\gg t^{\varepsilon}$, then, in view of \eqref{K estimates},
$H^-(\frac{mN}{M^2q^2})$ is negligible.
For the term in \eqref{eqn: use voronoi gl2} involving $H^+$, with the help of \eqref{J estimates},
we may replace it by
\begin{multline}\label{leading term}
  \frac{N^{3/4-it}}{M^{1/2}q^{1/2}\tau(\bar{\chi})}\sum_{\substack{u \bmod M \\u\neq b}}\bar{\chi}(u\ell)\sum_{\eta_2=\pm 1}\sum_{m\geq 1}
  \frac{\lambda_f(m)}{m^{1/4}}e\left(\frac{m\overline{\ell(aM+(b-u)q)}}{Mq}\right)
  \\
  \cdot \int_{\mathbb{R}}\xi^{-1/4}V(\xi)e\left(-\frac{t\log \xi}{2\pi}+\eta_2\frac{2\sqrt{mN\xi}}{Mq}-\frac{N\ell x\xi}{MqQ}\right)\dd \xi.
\end{multline}
Note that we have $\ell\asymp L$, $|x|\asymp X$ and $q\asymp R$.
By Lemma \ref{Huang lem 4.1} and Lemma \ref{GL2 Hpm} and according to
the size of $\frac{N\ell x}{MqQ}$, $\frac{n_1^2n_2N\ell}{q^3M^3r}$ and $\frac{mN}{M^2q^2}$,
we can
reduce $S_1^\pm(N,X,R)$ to the following three cases:

$\mathbf{Case\ (a):}$
\begin{align}\label{case:a}
  \bm{\frac{NLX}{MRQ}\asymp \left(\frac{n_1^2n_2NL}{R^3M^3r}\right)^{1/3}\gg t^\varepsilon}, \quad
  \bm{\frac{mN}{M^2R^2}\gg t^\varepsilon.}
\end{align}
In this case, we insert \eqref{eqn: use voronoi gl3} and \eqref{eqn: use voronoi gl2}
into \eqref{S11 pm} and use Lemma \ref{Huang lem 4.1} (i) and \eqref{leading term}, so that
it is sufficient to estimate
\begin{align}\label{eqn: goal of case a}
 \begin{split}
  &\frac{N^{5/4-it}}{\tau(\bar{\chi})M^{2}LQr^{1/2}} \sum_{\ell\in\mathcal{L}}
  \overline{A(1,\ell)\chi(\ell)}\ell^{1/2}
  \; \sideset{}{^\star}\sum_{b\bmod M}
  \sum_{\substack{q\sim R \\ (q,\ell M)=1}}\frac{1}{q^2}
  \;\sideset{}{^\star}\sum_{a\bmod{q}}
  \;\sum_{\substack{u\bmod M \\u\neq b}}\bar{\chi}(u)
  \\
  & \cdot
  \sum_{m\geq 1}
  \frac{\lambda_f(m)}{m^{1/4}}e\left(\frac{m\overline{\ell(aM+(b-u)q)}}{Mq}\right)
  \sum_{\eta_1,\eta_2=\pm1}\sum_{n_1|qMr}
  \sum_{n_2\asymp \frac{N_0}{n_1^2}} \frac{A(n_1,n_2)}{n_2^{1/2}}
  \\
  & \cdot
  S(r\overline{(aM+bq)},\eta_1 n_2;qMr/n_1)\int_{\mathbb{R}}\xi^{-1/4}V(\xi)
  e\left(-\frac{it\log \xi}{2\pi}+\eta_2\frac{2\sqrt{mN\xi}}{Mq}\right)
  \\
  & \cdot
  \int_\mathbb{R}g(q,x)
  e\left(-\frac{N\ell x\xi}{MqQ} +\eta_1 \frac{2(n_1^2n_2Q)^{1/2}}{Mq((-\eta_1 rx))^{1/2}}\right)
  \mathcal{W}\left(\frac{Q^{3/2}(n_1^2n_2)^{1/2}}{ r^{1/2}(-\eta_1 x)^{3/2}N\ell}\right) U\left(\frac{-\eta_1 x}{X}\right)\dd x \dd \xi,
 \end{split}
\end{align}
where $N_0=\frac{N^2L^2X^3r}{Q^3}$.
Let $x=-\eta_1 Xv$. Then the resulting $x$-integral becomes
\begin{align}\label{v-integral}
  & -\eta_1 X\int_{\mathbb{R}}e\left(\eta_1\frac{N\ell X\xi v}{MqQ} + \eta_1\frac{2(n_1^2n_2Q)^{1/2}}{Mq(rXv)^{1/2}}\right)
  g\left(q,-\eta_1 Xv\right)U(v)
  W\left(\frac{Q^{3/2}(n_1^2n_2)^{1/2}}{ r^{1/2}(Xv)^{3/2}N\ell}\right)\dd v.
\end{align}
Let
\[
  h(v) = \eta_1\frac{N\ell X\xi v}{MqQ} + \eta_1\frac{2(n_1^2n_2Q)^{1/2}}{Mq(rXv)^{1/2}}.
\]
Then
\[
  h'(v) = \eta_1\frac{N\ell X\xi}{MqQ} - \eta_1\frac{(n_1^2n_2Q)^{1/2}}{Mq(rX)^{1/2}}v^{-3/2}, \quad
  h''(v) = \eta_1\frac{3(n_1^2n_2Q)^{1/2}}{2Mq(rX)^{1/2}}v^{-5/2}.
\]
Note that the solution of $h'(v_0)=0$ is
$v_0 = \frac{(n_1^2n_2)^{1/3}Q}{r^{1/3}(N\ell\xi)^{2/3}X}\asymp1$, and
\[
  h(v_0) =  \eta_1\frac{3(n_1^2n_2N\ell\xi)^{1/3}}{r^{1/3}Mq},
  \quad
  h''(v_0) = \frac{3\eta_1}{2v_0^2} \cdot \frac{(n_1^2n_2Q)^{1/2}}{Mq(rXv_0)^{1/2}} = \frac{3\eta_1}{2v_0^2}\cdot\frac{ (n_1^2n_2N\ell\xi)^{1/3}}{r^{1/3}Mq}.
\]
By the argument below Lemma \ref{lemma:delta}, we can think $g(q,x)$ as a nice function
which satisfies
\begin{align}\label{eqn:good property for g}
  \frac{\partial^j}{\partial x^j} g(q,x) \ll Q^{\varepsilon_1} |x|^{-j},
\end{align}
up to a negligible error.
Here $\varepsilon_1$ is a small positive number such that $\frac{t^\varepsilon}{Q^{2\varepsilon_1}}\gg t^{\varepsilon/2}$.
Then, by applying
Lemma \ref{lemma:stationary_phase},
we have \eqref{v-integral} is equal to
\begin{align*}
   \frac{r^{1/6}(qM)^{1/2} X }{(n_1^2n_2N\ell\xi)^{1/6}}e\left(\eta_1\frac{ 3(n_1^2n_2N\ell\xi)^{1/3}}{r^{1/3}Mq}\right)g(q,-\eta_1Xv_0)\mathcal{U}(v_0)W\left(\frac{Q^{3/2}(n_1^2n_2)^{1/2}}{ r^{1/2}(Xv_0)^{3/2}N\ell}\right) + O(t^{-A}),
\end{align*}
where $\mathcal{U}$ is a certain compactly supported $1$-inert function depending on $A$.
We may assume $(n_1,M)=1$, since otherwise we have $M\mid n_1$ which leads to a simpler case.
Hence, by letting $\mathcal{V}(\xi)=\xi^{-5/12}V(\xi)g(q,-\eta_1Xv_0)\mathcal{U}(v_0)W(\frac{Q^{3/2}(n_1^2n_2)^{1/2}}{ r^{1/2}(Xv_0)^{3/2}N\ell})$,
at the cost of a negligible error, we can rewrite \eqref{eqn: goal of case a} as
\begin{multline}\label{eqn: re goal of case a}
  \frac{N^{13/12-it} X}{\tau(\bar{\chi})M^{3/2}LQr^{1/3}} \sum_{\ell\in\mathcal{L}}
  \overline{A(1,\ell)\chi(\ell)}\ell^{1/3}\sum_{\substack{q\sim R \\ (q,\ell M)=1}}\frac{1}{q^{3/2}}
  \\
   \cdot \sum_{\eta_1,\eta_2=\pm1}\sum_{n_1|qr} \frac{1}{n_1^{1/3}} \sum_{n_2\asymp \frac{N_0}{n_1^2}} \frac{A(n_1,n_2)}{n_2^{2/3}}
  \sum_{m\geq 1}\frac{\lambda_f(m)}{m^{1/4}}
  \mathcal{C}(m,n_1,n_2,\ell,q)\mathcal{J}_\mathbf{a}(m,n_1,n_2,\ell,q),
\end{multline}
where
\begin{align*}
  \mathcal{J}_\mathbf{a}(m,n_1,n_2,\ell,q)=\int_{\mathbb{R}}\mathcal{V}(\xi)
  e\left(-\frac{t}{2\pi}\log \xi + \eta_1\frac{3(n_1^2n_2N\ell \xi)^{1/3}}{r^{1/3}Mq}+\eta_2\frac{2\sqrt{mN\xi}}{Mq}\right)\dd \xi,
\end{align*}
and
\begin{multline}\label{mathcal C}
  \mathcal{C}(m,n_1,n_2,\ell,q)=\; \sideset{}{^\star}\sum_{b\bmod M}\,\sideset{}{^\star}\sum_{a \bmod q } S(r\overline{(aM+bq)},\eta_1n_2,qMr/n_1)
  \\
  \cdot \sum_{\substack{u \bmod M \\ u\neq b}}\bar{\chi}(u)
  e\left(\frac{m\overline{\ell(aM+(b-u)q)}}{Mq}\right).
\end{multline}
By partial integration, one can truncate the $m$-sum at $m\ll \max\{\frac{t^2R^2M^2}{N}, \frac{NL^2X^2}{Q^2}\}$.
We have
\[
  \mathcal{C}(m,n_1,n_2,\ell,q)=\sideset{}{^\star}\sum_{\alpha \bmod qMr/n_1 }f(\alpha,m\bar{\ell},q)
  \tilde{S}(\alpha,m\bar{\ell},q)e\left(\eta_1\frac{\bar{\alpha}n_1n_2}{qMr}\right),
\]
where
\[
  \tilde{S}(\alpha,m,q)=\;\sideset{}{^\star}\sum_{b \bmod M } \sum_{\substack{u \bmod M \\u\neq b}}\bar{\chi}(u)e\left(\frac{\bar{q}^2(n_1\alpha\bar{b}+m\overline{(b-u})}{M}\right),
\]
and
\[
  f(\alpha,m,q)=\sum_{\substack{d|q\\ n_1\alpha\equiv-m(\bmod d)}}d\mu(q/d).
\]

$\mathbf{Case\ (b):}$
\begin{align}\label{case:b}
  \bm{\frac{NLX}{MRQ}\asymp \left(\frac{n_1^2n_2NL}{R^3M^3r}\right)^{1/3}\asymp t}, \quad
  \bm{\frac{mN}{M^2R^2}\ll t^\varepsilon.}
\end{align}
In this case, we replace $H^\pm(z)$ by $\mathcal{I}(z)$ as defined in \eqref{eqn:mathcal(I)}. Hence, we are led to estiamte
 \begin{align*}
  &\frac{N^{3/2-it}}{\tau(\bar{\chi})M^{5/2}LQr^{1/2}} \sum_{\ell\in\mathcal{L}}
  \overline{A(1,\ell)\chi(\ell)}\ell^{1/2}
  \; \sideset{}{^\star}\sum_{b\bmod M}
  \sum_{\substack{q\sim R \\ (q,\ell M)=1}}\frac{1}{q^{5/2}}
  \;\sideset{}{^\star}\sum_{a\bmod{q}}
  \;\sum_{\substack{u\bmod M \\u\neq b}}\bar{\chi}(u)
    \nonumber
  \\
  & \cdot
  \sum_{m\geq 1}
  \frac{\lambda_f(m)}{m^{1/4}}e\left(\frac{m\overline{\ell(aM+(b-u)q)}}{Mq}\right)
  \sum_{\eta_1=\pm1}\sum_{n_1|qMr}
  \sum_{n_2\asymp \frac{N_0}{n_1^2}} \frac{A(n_1,n_2)}{n_2^{1/2}}\nonumber
  \\
  & \cdot
  S(r\overline{(aM+bq)},\eta_1 n_2;qMr/n_1)\int_{\mathbb{R}}\xi^{-1/4}V(\xi)J_f\left(\frac{mN\xi}{M^2q^2}\right)
  e\left(-\frac{it\log \xi}{2\pi}{Mq}\right)\nonumber
  \\
  & \cdot
  \int_\mathbb{R}g(q,x)
  e\left(-\frac{N\ell x\xi}{MqQ} +\eta_1 \frac{2(n_1^2n_2Q)^{1/2}}{Mq((-\eta_1 rx))^{1/2}}\right)
  \mathcal{W}\left(\frac{Q^{3/2}(n_1^2n_2)^{1/2}}{ r^{1/2}(-\eta_1 x)^{3/2}N\ell}\right) U\left(\frac{-\eta_1 x}{X}\right)\dd x \dd \xi.
\end{align*}
By doing a similar treatment as in $\mathbf{Case\ (a)}$, one can equate the above with (up to a negligible error and another term with $M\mid n_1$)
\begin{multline}\label{eqn: re goal of case b}
  \frac{N^{4/3-it} X}{\tau(\bar{\chi})M^{2}LQr^{1/3}} \sum_{\ell\in\mathcal{L}}
  \overline{A(1,\ell)\chi(\ell)}\ell^{1/3}\sum_{\substack{q\sim R \\ (q,\ell M)=1}}\frac{1}{q^{2}}
  \\
   \cdot \sum_{\eta_1,\eta_2=\pm1}\sum_{n_1|qr} \frac{1}{n_1^{1/3}} \sum_{n_2\asymp \frac{N_0}{n_1^2}} \frac{A(n_1,n_2)}{n_2^{2/3}}
  \sum_{m\geq 1}\frac{\lambda_f(m)}{m^{1/4}}
  \mathcal{C}(m,n_1,n_2,\ell,q)\mathcal{J}_\mathbf{b}(m,n_1,n_2,\ell,q),
\end{multline}
where $\mathcal{C}$ is defined as in \eqref{mathcal C} and
\begin{align*}
  \mathcal{J}_\mathbf{b}(m,n_1,n_2,\ell,q)=\int_{\mathbb{R}}\xi^{-1/4}V(\xi)J_f\left(\frac{mN\xi}{M^2q^2}\right)
  e\left(-\frac{t}{2\pi}\log \xi + \eta_1\frac{3(n_1^2n_2N\ell \xi)^{1/3}}{r^{1/3}Mq}\right)\dd \xi.
\end{align*}

$\mathbf{Case\ (c):}$
\begin{align}\label{case:c}
  \bm{\frac{n_1^2n_2}{R^3M^3r}LN\ll t^\varepsilon},\quad
  \bm{\frac{NLX}{MRQ}\ll t^\varepsilon}, \quad
  \bm{\frac{mN}{M^2R^2}\gg t^\varepsilon}.
\end{align}

Since $\frac{NLX}{MRQ}\ll t^\varepsilon$, we first deal with the $\xi$-integral in \eqref{leading term}.
Making a change of variable $\xi \rightsquigarrow \xi^2$, we have
\[
   \mathcal{J}_\mathbf{c}(m,\ell,q) = 2 \int_{\mathbb{R}}\xi^{-1/2}V(\xi^2) e\left(-\frac{N\ell x\xi^2}{MqQ}\right)
   e\left(-\frac{t\log \xi}{\pi}+\eta_2\frac{2\sqrt{mN}}{Mq}\xi\right)\dd \xi.
\]
Let
\[h(\xi)=-\frac{t\log \xi}{\pi}+\eta_2\frac{2\sqrt{mN}}{Mq}\xi.\]
Then we have
\[h'(\xi) = -\frac{t}{\pi\xi}+\eta_2\frac{2\sqrt{mN}}{Mq}, \quad
h''(\xi)=\frac{t}{\pi\xi^2}, \quad h^{(j)}(\xi)\asymp_j t, \quad j\geq2. \]
Note that $\frac{t}{1+(NLX/MRQ)^2}\gg t^{1-\varepsilon}$.
Hence, by Lemma \ref{lemma:stationary_phase}, the integral is
negligibly small unless $\frac{mN}{M^2 R^2}\asymp t$ and $\eta_2=1$, in which case
we have the stationary phase point $\xi_0=\frac{tMq}{2\pi\sqrt{mN}}$ and
\[
  \mathcal{J}_\mathbf{c}(m,\ell,q) = \frac{1}{t^{1/2}} e\left( -\frac{t}{\pi}\log \frac{tMq}{2\pi e\sqrt{mN}} \right) V_{x,\ell}\left(\frac{tMq}{\sqrt{mN}}\right) + O(t^{-A}),
\]
where $V_{x,\ell}$ is a $t^\varepsilon$-inert function.

Together with \eqref{eqn: use voronoi gl3} and \eqref{leading term}, we have $S_{11}^{\pm}(N,X,R)$ is equal to (up to a negligibly small error term and another term with $u=b$)
\begin{align*}
 \begin{split}
    &  \frac{1}{L^*} \sum_{\ell\in\mathcal{L}} \overline{A(1,\ell)} \int_\mathbb{R}
  \frac{1}{M}  \; \sideset{}{^\star}\sum_{b\bmod M}
  \frac{1}{Q}\sum_{\substack{q\sim R \\ (q,\ell M)=1}}
  \frac{1}{q} \; \sideset{}{^\star}\sum_{a\bmod q} g(q,x)U\left(\frac{\pm x}{X}\right)
  qM \sum_{\eta_1=\pm 1}\sum_{n_1|qMr} \\
  & \hskip 15pt \cdot
  \sum_{n_2=1}^\infty \frac{A(n_1,n_2)}{n_1n_2} S(r\overline{(aM+bq)},\eta_1n_2;qMr/n_1)  \left(\frac{n_1^2n_2 \ell N}{q^3M^3r}\right)^{1/2} W_{x,\ell}^{\sgn(\eta_1)}\left(\frac{n_1^2n_2}{q^3M^3r}\right) \\
  & \hskip 15pt \cdot
  \frac{N^{3/4-it}}{M^{1/2}q^{1/2}\tau(\bar{\chi})}\sum_{\substack{u \bmod M \\u\neq b}}\bar{\chi}(u\ell) \sum_{m\geq 1}
  \frac{\lambda_f(m)}{m^{1/4}}e\left(\frac{m\overline{\ell(aM+(b-u)q)}}{Mq}\right)
  \\
  & \hskip 15pt \cdot \frac{1}{t^{1/2}} e\left( -\frac{t}{\pi}\log \frac{tMq}{2\pi e\sqrt{mN}} \right) V_{x,\ell}\left(\frac{tMq}{\sqrt{mN}}\right) \mathrm{d}x.
 \end{split}
\end{align*}
We assume $(n_1,M)=1$, since otherwise we have $M \mid n_1$ which leads to a simpler case. Rearranging the sums, inserting a dyadic partition for the $n_2$-sum, and estimating the $x$-integral trivially, the above is bounded by
\begin{align*}
 \begin{split}
    & N^\varepsilon \sup_{1\ll N_0 \ll \frac{R^3 M^3 r}{LN}t^\varepsilon} \sup_{x\asymp X}  \big| S_{11}^\pm (N,X,R,N_0) \big|,
 \end{split}
\end{align*}
where
\begin{align*}
  & S_{11}^\pm (N,X,R,N_0) = \frac{N^{5/4}X }{ M^{5/2}LQ r^{1/2}}  \frac{1}{t^{1/2}}
  \sum_{\eta_1=\pm 1} \sum_{n_1 \leq Rr}  \sum_{n_2 \asymp \frac{N_0}{n_1^2}}  \frac{A(n_1,n_2)}{ n_2^{1/2}}  \sum_{\ell\in\mathcal{L}} \overline{A(1,\ell)\chi(\ell)} \ell^{1/2}  \\
  & \hskip 30pt \cdot
 \sum_{\substack{q\sim R \\ n_1\mid qr \\ (q,\ell M)=1}} \frac{1}{q^{2+2it}}\sum_{m\asymp \frac{R^2M^2t^2}{N}}
  \frac{\lambda_f(m)}{m^{1/4-it}} \mathcal{C}(m,n_1,n_2,\ell,q)
  W_{x,\ell}^{\sgn(\eta_1)}\left(\frac{n_1^2n_2}{q^3M^3r}\right)
   V_{x,\ell}\left(\frac{tMq}{\sqrt{mN}}\right),
\end{align*}
and $\mathcal{C}$ is defined as in \eqref{mathcal C}.

\section{Applying Cauchy and Poisson}\label{sec:cauchy+poisson}
\subsection{Case a}
In this subsection, we assume $\mathbf{Case\ (a)}$ which was defined in \eqref{case:a}.
Write $q=q_1q_2$ with $q_1|(rn_1)^\infty$ and $(q_2,rn_1)=1$, then we have
\begin{multline*}
  \eqref{eqn: re goal of case a}\ll \frac{N^{13/12+\varepsilon}X}{M^{2}LQr^{1/3}}
  \sum_{\eta_1,\eta_2=\pm 1}
  \sum_{n_1\ll Rr }\frac{1}{n_1^{1/3}}\sum_{\frac{n_1}{(n_1,r)}|q_1|(rn_1)^\infty}\frac{1}{q_1^{3/2}}
  \sum_{n_2\asymp \frac{N_0}{n_1^2}}
  \frac{|A(n_1,n_2)|}{n_2^{2/3}}
  \\
  \cdot \bigg|\sum_{\substack{\ell\in\mathcal{L}\\ (\ell,q_1)=1}}\overline{A(1,\ell)\chi(\ell)}\ell^{1/3}
  \sum_{\substack{q_2\sim R/q_1\\ (q_2,rn_1\ell M)=1}}\frac{1}{q_2^{3/2}}\sum_{m\ll \max\{\frac{t^2R^2M^2}{N}, \frac{NL^2X^2}{Q^2}\}}
  \frac{\lambda_f(m)}{m^{1/4}} \\
  \cdot
  \mathcal{C}(m,n_1,n_2,\ell,q_1q_2)\mathcal{J}_\mathbf{a}(m,n_1,n_2,\ell,q_1q_2)\bigg|.
\end{multline*}
Now we use the Cauchy--Schwarz inequality and \eqref{eqn:RS3-2} to get
\begin{align}\label{eqn:led to estimate}
  \ll \frac{N^{3/4+\varepsilon}X^{1/2}}{M^{2}L^{4/3}Q^{1/2}r^{1/2}}
  \sum_{\eta_1,\eta_2=\pm 1}
  \sup_{M_1\ll \max\{\frac{t^2R^2M^2}{N}, \frac{NL^2X^2}{Q^2}\}}
  \sum_{n_1\ll Rr}n_1^{\theta_3} \sum_{\frac{n_1}{(n_1,r)}|q_1|n_1^\infty}\frac{1}{q_1^{3/2}}\Omega_\mathbf{a}^{1/2},
\end{align}
where
\begin{multline*}
  \Omega_\mathbf{a}=\sum_{n_2\asymp \frac{N_0}{n_1^2}}\bigg|\sum_{\substack{\ell\in\mathcal{L}\\ (\ell,q_1)=1}}\overline{A(1,\ell)\chi(\ell)}\ell^{1/3}\sum_{\substack{q_2\sim R/q_1\\ (q_2,rn_1\ell M)=1}}\frac{1}{q_2^{3/2}}\sum_{m\sim M_1}\frac{\lambda_f(m)}{m^{1/4}}
  \\ \cdot
  \mathcal{C}(m,n_1,n_2,\ell,q_1q_2)\mathcal{J}_\mathbf{a}(m,n_1,n_2,\ell,q_1q_2)\bigg|^2.
\end{multline*}
Opening the absolute square, we get
\begin{multline*}
  \Omega_\mathbf{a}\ll\sum_{n_2\geq 1}W\left(\frac{n_1^2n_2}{N_0}\right)
  \underset{(\ell\ell',q_1)=1}{\sum_{\ell\in\mathcal{L}} \sum_{\ell'\in\mathcal{L}}} \overline{A(1,\ell)\chi(\ell)}
  A(1,\ell')\chi(\ell')(\ell\ell')^{1/3}
  \\
  \cdot \sum_{m\sim M_1}\frac{\lambda_f(m)}{m^{1/4}}
  \sum_{m'\sim M_1}\frac{\lambda_f(m')}{m'^{1/4}}
  \underset{(q_2q_2',rn_1M)=1}{\sum_{\substack{q_2\sim R/q_1\\(q_2,\ell)=1}}\sum_{\substack{q_2'\sim R/q_1\\(q_2',\ell')=1}}}\frac{1}{(q_2q_2')^{3/2}}
   \\
  \cdot \mathcal{C}(m,n_1,n_2,\ell,q_1q_2)\mathcal{J}_\mathbf{a}(m,n_1,n_2,\ell,q_1q_2)
  \overline{\mathcal{C}(m',n_1,n_2,\ell',q_1q_2')}\overline{\mathcal{J}_\mathbf{a}(m',n_1,n_2,\ell',q_1q_2')},
\end{multline*}
where $W$ is supported on $[1,2]$ and satisfies $W^{(j)}(x)\ll 1$.
We apply the Poisson summation formula on $n_2$, getting
\begin{align*}
  \Omega_\mathbf{a} \ll \frac{N_0q_1^{3}L^{2/3}}{n_1^2M_1^{1/2}R^{3}}
  \underset{(\ell\ell',q_1)=1}{\sum_{\ell\in\mathcal{L}}\sum_{\ell'\in\mathcal{L}}}
  |A(1,\ell)A(1,\ell')|   &
  \sum_{m\sim M_1}\sum_{m'\sim M_1}|\lambda_f(m)| |\lambda_f(m')| \\
  & \cdot \underset{(q_2q_2',rn_1M)=1}{\sum_{\substack{q_2\sim R/q_1\\(q_2,\ell)=1}}\sum_{\substack{q_2'\sim R/q_1\\(q_2',\ell')=1}}}
  \sum_{n_2\geq 1}|\mathfrak{C}(n_2)||\mathfrak{J}_\mathbf{a}(n_2)|,
\end{align*}
where
\begin{multline}\label{eqn: character sum case a}
  \mathfrak{C}(n_2)=\;\sideset{}{^\star}\sum_{b\bmod M }
  \;\sideset{}{^\star}\sum_{b'\bmod M }
  \Bigg(\sum_{\substack{u\bmod M \\ u\neq b}}\bar{\chi}(u)e\left(\frac{m\overline{q_1^2q_2^2\ell(b-u)}}{M}\right)\Bigg)
  \\ \cdot
  \Bigg(\sum_{\substack{u'\bmod M \\ u'\neq b'}}\chi(u')e\left(\frac{-m'\overline{q_1^2q_2'^2\ell'(b'-u')}}{M}\right)\Bigg) \Bigg(\sum_{d|q_1q_2}\sum_{d'|q_1q_2'}dd'\mu(q_1q_2/d)\mu(q_1q_2'/d')
  \\
  \cdot
  \sideset{}{^\star}
  \sum_{\alpha(\bmod Mrq_1q_2/n_1)}\sideset{}{^\star}\sum_{\substack{\alpha'(\bmod Mrq_1q_2'/n_1)\\q_2'\bar{\alpha}-q_2\bar{\alpha'}\equiv-\eta_1n_2(Mrq_1q_2q_2'/n_1)
  \\n_1\alpha\equiv-m\bar{\ell}(d)
  \\n_1\alpha'\equiv-m'\bar{\ell'}(d')}}
  e\left(\frac{n_1\alpha\overline{bq_1^2q_2^2}-n_1\alpha'\overline{b'q_1^2q_2'^2}}{M}\right)\Bigg),
\end{multline}
and
\begin{align*}
  \mathfrak{J}_\mathbf{a}(n_2)=\int_{\mathbb{R}}W(w)
  \mathcal{I}_\mathbf{a}(N_0w,m,q_2)
  \overline{\mathcal{I}_\mathbf{a}(N_0w,m',q_2')}
  e\left(-\frac{N_0n_2w}{q_1q_2q_2'Mn_1r}\right) \dd w
\end{align*}
with
\begin{align*}
  \mathcal{I}_\mathbf{a}(w,n,q_2)=\int_{\mathbb{R}}\mathcal{V}(\xi)
  e\left(-\frac{t}{2\pi}\log \xi + \eta_1\frac{3(wN\ell \xi)^{1/3}}{r^{1/3}Mq_1q_2}+\eta_2\frac{2\sqrt{mN\xi}}{Mq_1q_2}\right)\dd \xi.
\end{align*}

\subsubsection{$\bm{\frac{NLX}{MRQ}\ll t^{1-\varepsilon}}$}
We first consider
$\mathcal{I}(N_0w,m,q_2)$.
Let
\begin{align}\label{h(xi)}
  g(\xi)=-\frac{t}{2\pi}\log \xi+\eta_1\frac{3(N_0wN\ell \xi)^{1/3}}{r^{1/3}Mq_1q_2}+\eta_2\frac{2\sqrt{mN\xi}}{Mq_1q_2}.
\end{align}
There exists a stationary phase point $\xi_*$ if and only if $m\asymp\frac{t^2M^2R^2}{N}$
and $\eta_2=1$, in which case
$\xi_*$ can be written as $\xi_0+\xi_1+\xi_2+\cdots$ with
\begin{align*}
  & \xi_0=\frac{t^2M^2q_1^2q_2^2}{4\pi^2 mN}=\left(\frac{t}{\pi C}\right)^2\asymp1,
  \\
  &\xi_1=-\eta_1\frac{4\pi Bw^{1/3}}{3t}\xi_0^{4/3}
  \asymp \frac{B}{t},
  \\
  &\xi_2=\frac{28\pi^2 B^2w^{2/3}}{27t^2}\xi_0^{5/3}
  \asymp \frac{B^2}{t^2},
  \\
  & \xi_i=f_i(t,C)\left(\eta_1\frac{Bw^{1/3}}{t}\right)^i\ll \left(\frac{B}{t}\right)^i, \quad i\geq3,
\end{align*}
where $B=\frac{3(N_0N\ell)^{1/3}}{r^{1/3}Mq_1q_2}\asymp\frac{NLX}{MRQ}$, $C=\frac{2\sqrt{mN}}{Mq_1q_2}$
and $f_i(t,C)\asymp1$ is a function.
Recall  that $\mathcal{V}(\xi)=\xi^{-5/12}V(\xi)g(q,-\eta_1Xv_0)\mathcal{U}(v_0)W(\frac{Q^{3/2}(n_1^2n_2)^{1/2}}{ r^{1/2}(Xv_0)^{3/2}N\ell})$, $v_0 = \frac{(n_1^2n_2)^{1/3}Q}{r^{1/3}(N\ell\xi)^{2/3}X}\asymp1$
and \eqref{eqn:good property for g}. So it is easy to check the conditions in Lemma \ref{lemma:stationary_phase}. By using this lemma together with the Taylor expansion,
$\mathcal{I}_\mathbf{a}(N_0w,m,q_2)$ is essentially reduced to
\begin{align}\label{xi stationary}
  \frac{1}{t^{1/2}}\xi_0^{-it}e\left(Bw^{1/3}g_1(C)+B^2w^{2/3}g_2(C)+O\left(\frac{|B|^3}{t^2}\right)\right),
\end{align}
where $g_1(C)=\eta_1\xi_0^{1/3}=\eta_1\frac{t^{2/3}}{(\pi C)^{2/3}}\asymp 1$ and $g_2(C)= -\frac{4\pi }{9t}\xi_0^{2/3}\ll \frac{1}{t}$.
To estimate $\mathfrak{J}_a(n_2)$, we use the strategy in \cite[Lemma 4.3]{LinSun}
and \cite[Lemma 5]{Munshi2018} to get the following result.
\begin{lemma}\label{lemma:J}
Let $N_2=\frac{Q^2Rn_1}{NLX^2q_1}t^\varepsilon$ and
$N_2'=t^\varepsilon(\frac{NLn_1}{M^2Rt^2q_1}+\frac{R^2Q^3Mn_1}{N^2L^2X^3q_1})$. Assume $\frac{NLX}{MRQ}\ll t^{1-\varepsilon}$.
  \begin{itemize}
    \item [(i)]  We have
    $\mathfrak{J}_\mathbf{a}(n_2) \ll t^{-A}$ unless $n_2\ll N_2$, in which case one has
    \begin{align}\label{eqn:J-i}
      \mathfrak{J}_a(n_2) \ll \frac{1}{t^{1-\varepsilon}}.
    \end{align}
    \item [(ii)] If $N_2'\ll n_2\ll N_2$, we have
    \begin{align}\label{eqn:J-ii}
      \mathfrak{J}_\mathbf{a}(n_2) \ll \frac{RQ^{3/2}M^{1/2}n_1^{1/2}}{t^{1-\varepsilon}NLX^{3/2}q_1^{1/2}n_2^{1/2}}.
    \end{align}
    \item [(iii)] If $q_2=q_2'$, we have $\mathfrak{J}_\mathbf{a}(0)\ll t^{-A}$ unless $\ell m'-\ell'm\ll t^\varepsilon\left(\frac{M_1N^2L^3X^2}{M^2R^2Q^2t^2}+\frac{M_1MRQ}{NX}\right)$.
  \end{itemize}
\end{lemma}
\begin{proof}
    Let $w=u^3$. Then we may equate the $w$-integral in $\mathfrak{J}_\mathbf{a}$ with
  \begin{multline*}
    \int_{\mathbb{R}}W(u^3)u^2
    \\
   \cdot e\left(-\frac{N_0n_2u^3}{q_1q_2q_2'Mn_1r}
   +(Bg_1(C)-B'g_1(C'))u+(B^2g_2(C)-B'^2g_2(C'))u^2+O\left(\frac{B^3}{t^2}\right)\right)\dd u,
  \end{multline*}
where $B'=\frac{3(N_0N\ell')^{1/3}}{r^{1/3}Mq_1q_2'}$, $C'=\frac{2\sqrt{m'N}}{Mq_1q_2'}$.
Applying integration by parts, the above integral is $\ll t^{-A}$ if $n_2\gg N_2$, which gives the
first result in (i).
The second result in (i) is obvious, since we may save $t^{1/2}$ in both $\mathcal{I}_\mathbf{a}(N_0w,m,q_2)$ and $\overline{\mathcal{I}_\mathbf{a}(N_0w,m',q_2')}$ according to \eqref{xi stationary}.

It is easy to see that
\begin{multline}\label{eqn:compare the co between u and u2}
   B^2g_2(C)-B'^2g_2(C')
   \ll \frac{B\xi_0^{1/3}+B'\xi_0'^{1/3}}{t}|B\xi_0^{1/3}-B'\xi_0'^{1/3}|
  \ll |Bg_1(C)-B'g_1(C')|t^{-\varepsilon},
\end{multline}
where we have used $\xi_0'=(\frac{t}{\pi C'})^2\asymp 1$ and $B$, $B'\asymp \frac{NLX}{MRQ}\ll t^{1-\varepsilon}$.
Therefore, if $N_2'\ll n_2\ll N_2$, the $u$-integral is $O(t^{-A})$ unless
$|Bg_1(C)-B'g_1(C')|\asymp \frac{N_0n_2}{q_1q_2q_2'Mn_1r}$.
By the second derivative test and \eqref{xi stationary}, we get
\eqref{eqn:J-ii}.

For $n_2=0$ and $q_2=q_2'$, we may rewrite the above $u$-integral as
\begin{align*}
  \int_{\mathbb{R}}W(u^3)u^2
    e\left((Bg_1(C)-B'g_1(C'))u+(B^2g_2(C)-B'^2g_2(C'))u^2+O\left(\frac{B^3}{t^2}\right)\right)\dd u.
\end{align*}
Notice that $\frac{Bg_1(C)}{(m'\ell)^{1/3}}=\frac{B'g_1(C')}{(m\ell')^{1/3}}$,
and $Bg_1(C)-B'g_1(C')=\frac{Bg_1(C)}{(m'\ell)^{1/3}}((m'\ell)^{1/3}-(m\ell')^{1/3})$.
So by partial integration and \eqref{eqn:compare the co between u and u2}, the $u$-integral is $O(t^{-A})$ unless
\begin{align*}
  (m'\ell)^{1/3}-(m\ell')^{1/3}\ll \left(\frac{B^3}{t^2}+1\right)\frac{(M_1L)^{1/3}t^\varepsilon}{B}.
\end{align*}
This actually proves the result in (iii).
%
%
\end{proof}

%
%

\subsubsection{$\bm{\frac{NLX}{MRQ}\gg t^{1-\varepsilon}}$}
It is easy to see that $R\ll \frac{N^{1+\varepsilon}LX}{MtQ}$.
We have the following Lemma \ref{lemma:J2}.
\begin{lemma}\label{lemma:J2}
Let $N_2$ be defined as in Lemma \ref{lemma:J}.
Then, if $\frac{NLX}{MRQ}\gg t^{1-\varepsilon}$, one has the following estimates.
  \begin{itemize}
    \item [(i)]  If $n_2\gg N_2$, we have
    $\mathfrak{J}_\mathbf{a}(n_2)\ll N^{-A}$.
    \item [(ii)] If $n_2\ll N_2$, we have
    \begin{align*}
      \mathfrak{J}_\mathbf{a}(n_2)\ll \frac{MRQ}{N^{1-\varepsilon}LX}.
    \end{align*}
    \end{itemize}
\end{lemma}
\begin{proof}
  The first result can be done by applying integration by parts with respect to the $w$-integral.
  For $n_2\ll N_2$, we can use the arguments as in \cite[Lemma 1]{Munshi2018} to see
  \begin{align*}
    \int_{\mathbb{R}}W(w)|\mathcal{I}_\mathbf{a}(N_0w,n,\ell,q_2)|^2 \dd w \ll \frac{MRQ}{N^{1-\varepsilon}LX},
  \end{align*}
  which implies (ii).
 \end{proof}
\begin{remark}
 In the case of $\frac{NLX}{MRQ}\gg t^{1+\varepsilon}$, we remark that one may replace it by a more explicit version like Lemma \ref{lemma:J}.
 However, the present result is enough for our purpose.
\end{remark}
\subsection{Case b}
After a similar treatment, and noting that $m\ll \frac{M^2R^2t^\varepsilon}{N}$, we have
\begin{align}\label{eqn:led to estimate case b}
  \eqref{eqn: re goal of case b}\ll \frac{N^{1+\varepsilon}X^{1/2}}{M^{5/2}L^{4/3}Q^{1/2}r^{1/2}}
  \sum_{\eta_1=\pm 1}
  \sup_{M_1\ll \frac{M^2R^2t^\varepsilon}{N}}
  \sum_{n_1\ll Rr}n_1^{\theta_3} \sum_{\frac{n_1}{(n_1,r)}|q_1|n_1^\infty}\frac{1}{q_1^{2}}\Omega_\mathbf{b}^{1/2},
\end{align}
where
\begin{multline*}
  \Omega_\mathbf{b}\ll \frac{N_0q_1^{4}L^{2/3}}{n_1^2M_1^{1/2}R^{4}}
  \underset{(\ell\ell',q_1)=1}{\sum_{\ell\in\mathcal{L}}\sum_{\ell'\in\mathcal{L}}}
  |A(1,\ell)A(1,\ell')|
  \sum_{m\sim M_1}\sum_{m'\sim M_1}|\lambda_f(m)||\lambda_f(m')|
  \\ \cdot \underset{(q_2q_2',n_1 M)=1}{\sum_{\substack{q_2\sim R/q_1\\(q_2,\ell)=1}}\sum_{\substack{q_2'\sim R/q_1\\(q_2',\ell')=1}}}
  \sum_{n_2\geq 1}|\mathfrak{C}(n_2)||\mathfrak{J}_\mathbf{b}(n_2)|,
\end{multline*}
with $\mathfrak{C}(n_2)$ defined as in \eqref{eqn: character sum case a} and
\begin{align*}
  \mathfrak{J}_\mathbf{b}(n_2)=\int_{\mathbb{R}}W(w)
  \mathcal{I}_\mathbf{b}(N_0w,m,q_2)\overline{\mathcal{I}_\mathbf{b}(N_0w,m',q_2')}
  e\left(-\frac{N_0n_2w}{q_1q_2q_2'Mn_1r}\right) \dd w,
\end{align*}
\begin{align*}
  \mathcal{I}_\mathbf{b}(w,n,q_2)=\int_{\mathbb{R}}\xi^{-1/4}V(\xi)J_f\left(\frac{mN\xi}{M^2q^2}\right)
  e\left(-\frac{t}{2\pi}\log \xi + \eta_1\frac{3(wN\ell \xi)^{1/3}}{r^{1/3}Mq_1q_2}\right)\dd \xi.
\end{align*}
By the exactly treatment, we have the following lemma.
\begin{lemma}\label{lemma:J3}
The results in Lemma \ref{lemma:J2} hold when replacing $\mathfrak{J}_\mathbf{a}$
by $\mathfrak{J}_\mathbf{b}$.
\end{lemma}
\subsection{Case c}
After a similar treatment, we have
\begin{align*}
  & S_{11}^\pm (N,X,R,N_0) \leq \frac{N^{5/4}X }{ M^{5/2}LQ r^{1/2}}  \frac{1}{t^{1/2}}
  \sum_{\eta_1=\pm 1} \sum_{n_1 \leq Rr}
   \sum_{\frac{n_1}{(n_1,r)}\mid q_1\mid (rn_1)^\infty} \frac{1}{q_1^2}
     \\
  & \hskip 30pt \cdot
  \sum_{n_2 \asymp \frac{N_0}{n_1^2}}  \frac{|A(n_1,n_2)|}{ n_2^{1/2}}
  \bigg| \sum_{\ell\in\mathcal{L}} \overline{A(1,\ell)\chi(\ell)} \ell^{1/2}
  \sum_{\substack{q_2\sim R/q_1 \\  (q_2,\ell M rn_1)=1}} \frac{1}{q_2^{2+2it}}
  \sum_{m\asymp \frac{R^2M^2t^2}{N}}
  \frac{\lambda_f(m)}{m^{1/4-it}} \\
  & \hskip 30pt \cdot \mathcal{C}(m,n_1,n_2,\ell,q)
  W_{x,\ell}^{\sgn(\eta_1)}\left(\frac{n_1^2n_2}{q^3M^3r}\right)
   V_{x,\ell}\left(\frac{tMq}{\sqrt{mN}}\right) \bigg|.
\end{align*}
By Cauchy--Schwarz inequality and \eqref{eqn:RS3-2} we have
\begin{align}\label{eqn:led to estimate c}
  & S_{11}^\pm (N,X,R,N_0) \ll \frac{N^{5/4+\varepsilon}X }{ M^{5/2}LQ r^{1/2}}  \frac{1}{t^{1/2}}
  \sum_{\eta_1=\pm 1} \sum_{n_1 \leq Rr} n_1^{\theta_3}
   \sum_{\frac{n_1}{(n_1,r)}\mid q_1\mid (rn_1)^\infty} \frac{1}{q_1^2}
   \Omega_\mathbf{c} ^{1/2},
\end{align}
where
\begin{align*}
  \Omega_\mathbf{c} & = \sum_{n_2 \asymp \frac{N_0}{n_1^2}}
  \bigg| \sum_{\ell\in\mathcal{L}} \overline{A(1,\ell)\chi(\ell)} \ell^{1/2}
  \sum_{\substack{q_2\sim R/q_1 \\  (q_2,\ell M rn_1)=1}} \frac{1}{q_2^{2+2it}}
  \sum_{m\asymp \frac{R^2M^2t^2}{N}}
  \frac{\lambda_f(m)}{m^{1/4-it}} \\
  & \hskip 90pt \cdot \mathcal{C}(m,n_1,n_2,\ell,q)
  W_{x,\ell}^{\sgn(\eta_1)}\left(\frac{n_1^2n_2}{q^3M^3r}\right)
   V_{x,\ell}\left(\frac{tMq}{\sqrt{mN}}\right) \bigg|^2 .
\end{align*}
Opening the square we get
\begin{align*}
  \Omega_\mathbf{c} & \ll L
   \sum_{\ell\in\mathcal{L}}  |A(1,\ell)|
  \sum_{\substack{q_2\sim R/q_1 \\  (q_2,\ell M rn_1)=1}} \frac{1}{q_2^{2}}
  \sum_{m\asymp \frac{R^2M^2t^2}{N}}
  \frac{|\lambda_f(m)|}{m^{1/4}} \\
  & \hskip 30pt \cdot
  \sum_{\ell'\in\mathcal{L}} |A(1,\ell')|
  \sum_{\substack{q_2'\sim R/q_1 \\  (q_2',\ell' M rn_1)=1}} \frac{1}{q_2'^{2}}
  \sum_{m'\asymp \frac{R^2M^2t^2}{N}}
  \frac{|\lambda_f(m')|}{m'^{1/4}}   \\
  & \hskip 30pt \cdot \bigg| \sum_{n_2 \geq1} W\left(\frac{n_1^2n_2}{N_0}\right)
  \mathcal{C}(m,n_1,n_2,\ell,q)
   \overline{\mathcal{C}(m',n_1,n_2,\ell',q')}
    \bigg| ,
\end{align*}
 where $W\left(\frac{n_1^2n_2}{N_0}\right)$ is a smooth compactly supported function which contains the weight function $W_{x,\ell}^{\sgn(\eta_1)}\left(\frac{n_1^2n_2}{q^3M^3r}\right)
  \overline{W_{x,\ell}^{\sgn(\eta_1)}\left(\frac{n_1^2n_2}{q^3M^3r}\right)}$.
  Note that by \eqref{eqn:W(z)} we have
\[ \frac{\partial^j}{\partial n_2^j} W\left(\frac{n_1^2n_2}{N_0}\right) \ll_j t^\varepsilon n_2^{-j}, \quad j\geq0. \]
By the Poisson summation formula modulo $Mrq_1q_2q_2'/n_1$ we get
\begin{align*}
  \Omega_\mathbf{c} & \ll \frac{N_0}{n_1^2}  \frac{L q_1^4}{R^4} \frac{N^{1/2}}{RMt}
   \sum_{\ell\in\mathcal{L}}  |A(1,\ell)|
  \sum_{\substack{q_2\sim R/q_1 \\  (q_2,\ell M rn_1)=1}}
  \sum_{m\asymp \frac{R^2M^2t^2}{N}}
    \\
  & \hskip 30pt \cdot
  \sum_{\ell'\in\mathcal{L}}|A(1,\ell')|
  \sum_{\substack{q_2'\sim R/q_1 \\  (q_2',\ell' M rn_1)=1}}
  \sum_{m'\asymp \frac{R^2M^2t^2}{N}} |\lambda_f(m')|^2  \sum_{n_2\in\mathbb{Z}}
  |\mathfrak{C}(n_2)| |\mathfrak{I}_\mathbf{c}(n_2)|,
\end{align*}
where $\mathfrak{C}(n_2)$ is defined as in \eqref{eqn: character sum case a}
and
\begin{align*}
  \mathfrak{I}_\mathbf{c}(n_2) &  = \frac{n_1^2} {N_0}\int_{\mathbb{R}}  W\left(\frac{n_1^2 u}{N_0}\right) e\left(-\frac{un_2}{Mrq_1q_2q_2'/n_1}\right) \dd u
  = \int_{\mathbb{R}}  W\left(\xi\right) e\left(-\frac{N_0 n_2 \xi}{Mrq_1q_2q_2' n_1}\right) \dd \xi.
\end{align*}
By repeated integration by parts we have
\begin{equation}\label{eqn:I-c}
  \mathfrak{I}_\mathbf{c}(n_2) \ll \left\{ \begin{array}{ll}
                                  t^{-A}, & \textrm{if $n_2 \gg \frac{MrR^2 n_1}{q_1 N_0}t^\varepsilon$,} \\
                                  t^\varepsilon, & \textrm{if $n_2 \ll \frac{MrR^2 n_1}{q_1 N_0}t^\varepsilon$.}
                                \end{array} \right.
\end{equation}

\section{The zero frequency}\label{sec:zero-freq}

In this section we estimate the contribution from the terms with $n_2=0$.
Denote the contribution of this part to $\Omega_*$ by $\Omega_{0}$, where $*\in\{\mathbf{a},\mathbf{b},\mathbf{c}\}$.
Note that $q_2'\bar{\alpha}-q_2\bar{\alpha'}\equiv 0 \; (\bmod \; Mq_2q_2')$.
So we have $q_2'=(q_2'\bar{\alpha},Mq_2q_2')=(q_2\bar{\alpha'},Mq_2q_2')=q_2$,
and hence $\alpha=\alpha'$. We have
\begin{multline*}
  \mathfrak{C}(0)= \delta_{q=q'}
  \;\sideset{}{^\star}\sum_{b\; (\mod M)}
  \;\sideset{}{^\star}\sum_{b'\; (\mod M)}
  \left(\sum_{\substack{u\; (\mod M)\\ u\neq b}}
  \bar{\chi}(u)e\left(\frac{m\overline{q^2\ell(b-u)}}{M}\right)\right)
  \\
  \cdot
  \left(\sum_{\substack{u'\; (\mod M)\\ u'\neq b'}}
  \chi(u')e\left(\frac{-m'\overline{q^2 \ell'(b'-u')}}{M}\right)\right)
  \sum_{d|q} \sum_{d'|q} dd'\mu(q/d)\mu(q/d')
  \\
  \cdot
  \sideset{}{^\star}\sum_{\substack{\alpha\; (\mod Mrq/n_1)
  \\n_1\alpha\equiv-m\bar{\ell}(d)
  \\n_1\alpha\equiv-m'\bar{\ell'}(d')}}
  e\left(\frac{n_1\alpha\overline{bq^2} -n_1\alpha\overline{b'q^2}}{M}\right).
\end{multline*}

\subsection{Case a: $t^\varepsilon \ll \frac{LNX}{MRQ} \ll t^{1-\varepsilon}$}
\label{subsec:zero-freq-generic}

\subsubsection{$M \mid (m\bar{\ell}-m'\bar{\ell'})$}
Denote the contribution of this part to $\Omega_0$ by $\Omega_{01}$.
Moreover, the $\alpha$-sum depends on either  $b\equiv b' \;(\mod M)$ or
 $b\not\equiv b' \;(\mod M)$.
The character sum becomes
\begin{align}\label{eqn:char_sum0}
  \mathfrak{C}(0) &
  \ll M |\mathfrak{C}_1'|\sum_{d|q}\sum_{d'|q}dd'
  \sum_{\substack{\alpha\; (\mod \frac{rq}{n_1})\\
  n_1 \alpha\equiv-m\bar{\ell}\; (\mod d)\\
  n_1 \alpha\equiv-m'\bar{\ell'}\; (\mod d')}}1
  +
  |\mathfrak{C}_1''| \sum_{d|q}\sum_{d'|q}  d d '
  \sum_{\substack{\alpha\; (\mod \frac{rq}{n_1})\\
  n_1 \alpha\equiv-m\bar{\ell}\; (\mod d)\\
  n_1 \alpha\equiv-m'\bar{\ell'}\; (\mod d')}}1  \nonumber
  \\
  & \ll (M |\mathfrak{C}_1'|+|\mathfrak{C}_1''|)
  \sum_{d|q}\sum_{d'|q} (d,d') \;
  rq \; \delta_{(d,d')\mid (m\ell'-m'\ell)},
\end{align}
where
\begin{align*}
  \mathfrak{C}_1'=\sideset{}{^\star}\sum_{b\; (\mod M)}
  \sum_{\substack{u\; (\mod M)\\u\neq b}}
  \sum_{\substack{u'\; (\mod M)\\u'\neq b}}
  \bar{\chi}(u)\chi(u')
  e\left(\frac{m\overline{q^2\ell(b-u)}}{M}\right)
  e\left(-\frac{m'\overline{q^2\ell'(b-u')}}{M}\right),
\end{align*}
and
\begin{align*}
  \mathfrak{C}_1''
  &=\sideset{}{^\star}\sum_{b\; (\mod M)}
  \sideset{}{^\star}\sum_{\substack{b'\; (\mod M) \\ b'\not\equiv b \; (\mod M)}}\sum_{\substack{u\; (\mod M)\\u\neq b}}
  \sum_{\substack{u'\; (\mod M)\\u'\neq b'}}
  \bar{\chi}(u)\chi(u') \\
  & \hskip 60pt \cdot
  e\left(\frac{m\overline{q^2\ell(b-u)}}{M}\right)
  e\left(-\frac{m'\overline{q^2\ell'(b'-u')}}{M}\right) .
\end{align*}
Since $M\mid (m\bar{\ell}-m'\bar{\ell'})$, similar to  \cite[Eq. (6.3)]{Sharma2019}, we have square root cancellation in the sum over $u$ and $u'$, and hence we obtain
\[
  \mathfrak{C}_1' \ll M^2
  \quad \textrm{and} \quad
  \mathfrak{C}_1'' \ll M^3.
\]
Hence
\[
  \mathfrak{C}(0)   \ll M^3
  \sum_{d|q}\sum_{d'|q} (d,d') \;
  rq \; \delta_{(d,d')\mid (m\ell'-m'\ell)}.
\]

Note that $(M,(d,d'))=1$
and
\[ |A(1,\ell)A(1,\ell')\lambda_f(m)\lambda_f(m')| \ll |A(1,\ell)\lambda_f(m')|^2 + |A(1,\ell')\lambda_f(m)|^2.
\]
By Lemma \ref{lemma:J}, we have
\begin{align} \label{eqn:sum-l&m}
  \sum_{\ell}\sum_{\ell'}\sum_{m}\sum_{m'} |\mathfrak{J}(0)|
  & \ll
  \sum_{\ell}|A(1,\ell)|^2 \sum_{m'} |\lambda_f(m')|^2 \mathop{\sum_{\ell'}\sum_{m}}_{M(d,d')\mid (m\ell'-m'\ell)} |\mathfrak{J}(0)|
  \nonumber \\
  & \hskip 30pt +
  \sum_{\ell'}|A(1,\ell')|^2 \sum_{m} |\lambda_f(m)|^2 \mathop{\sum_{\ell}\sum_{m'}}_{M(d,d')\mid (m\ell'-m'\ell)} |\mathfrak{J}(0)|
  \nonumber \\
  &
  \ll N^\varepsilon  L M_1 \left(\frac{ LM_1 (\frac{LNX}{MRQt})^2 + LM_1\frac{MRQ}{LNX}  }{M(d,d')} +1 \right) \frac{1}{t} .
\end{align}
Hence we have
\begin{align*}
  \Omega_{01}
  &
  \ll N^\varepsilon \frac{N_0  M^3}{n_1^2 M_1^{1/2}} \frac{L^{2/3} q_1^3}{R^3}
  \bigg(
  \sum_{\substack{q_2\sim R/q_1\\ (q_2,rn_1)=1}}
  \sum_{\substack{d|q \\d'|q}}
   rq   LM_1
   \left(\frac{LM_1 (\frac{LNX}{MRQt})^2 + LM_1\frac{MRQ}{LNX}   }{M} +  q \right)
   \frac{1}{t}  \bigg)
  \\
  &
  \ll N^\varepsilon \frac{N_0 M^3}{n_1^2 M_1^{1/2}}
    \frac{L^{2/3} q_1^2}{R^2}
   r R   LM_1
   \bigg(\frac{LM_1 (\frac{LNX}{MRQt})^2 + LM_1\frac{MRQ}{LNX}   }{M t} + \frac{R}{t}
  \bigg)
.
\end{align*}
By using $N_0=\frac{N^2L^2X^3r}{Q^3}$ and $M_1\ll \frac{t^2R^2M^2}{N}$, we get
\begin{align*}
  \Omega_{01} &
  \ll N^\varepsilon \frac{r^2 N^{3/2} L^{11/3} R q_1^2  M^4 X^3}{n_1^2 Q^3}
  \bigg(\frac{L^3 N  X^2}{M R Q^2}
  +  \frac{M^2R^2Q t^2}{ N^2X } + 1 \bigg).
\end{align*}
Hence, the contribution from $\Omega_{01}$ to \eqref{eqn:led to estimate} is
\begin{align*} 
  &\ll \frac{N^{3/4+\varepsilon}X^{1/2}}{M^{2}L^{4/3} Q^{1/2}r^{1/2}}
   \sum_{n_1\ll RMr} n_1^{\theta_3}\sum_{\frac{n_1}{(n_1,r)}|q_1|(rn_1)^\infty} \frac{1}{n_1 q_1^{1/2}}
  \left( \frac{r^2 N^{3/2} L^{11/3} R   M^4 X^3}{Q^3} \right)^{1/2}
   \nonumber
   \\
   & \hskip 30pt \cdot
  \left(\frac{L^3 N  X^2}{M R Q^2} +  \frac{M^2R^2Q t^2}{ N^2X } +  1  \right)  ^{1/2}
  \nonumber
  \\
    &
    \ll    N^{\varepsilon} r^{1/2} \frac{N^{3/2} L^{1/2} R^{1/2}   X^{2}}
    {  Q^{2}}
  \left(\frac{L^{3/2}  N^{1/2} X}{M^{1/2}  R^{1/2}Q}
  +  \frac{M R Q^{1/2} t }{ N X^{1/2} } +  1 \right).
\end{align*}
Recall $Q=(\frac{NL}{MK})^{1/2}$.
Thus, by $X\ll t^\varepsilon$ and $R\leq Q$, we arrive at
\begin{align} \label{eqn:Omega01<<}
    &
    \ll
     N^{\varepsilon} r^{1/2} \frac{N^{2} L^{2}  } { M^{1/2} Q^{3}}
     + N^{\varepsilon} r^{1/2}  N^{1/2} L^{1/2} Mt
     +  N^{\varepsilon} r^{1/2} \frac{N^{3/2} L^{1/2}  } {  Q^{3/2}}\nonumber
    \\
   &
   \ll N^{\varepsilon} r^{1/2}  N^{1/2} L^{1/2} M K^{3/2}
   + N^{\varepsilon} r^{1/2}  N^{1/2} L^{1/2} M t
   +
   N^{\varepsilon} r^{1/2} \frac{N^{3/4}  M^{3/4}K^{3/4} }
    { L^{1/4}}.
\end{align}

\subsubsection{$M \nmid(m\bar{\ell}-m'\bar{\ell'})$}
Denote the contribution of this part to $\Omega_0$ by $\Omega_{02}$.
In this case, we also have $q_2=q_2'$ and $\alpha=\alpha'$.
So we can estimate the character as in \eqref{eqn:char_sum0}.
Since $M \nmid(m\bar{\ell}-m'\bar{\ell'})$, the non degeneracy holds for the variables $b,u,u'$ in $\mathfrak{C}_1'$ and $\mathfrak{C}_1''$ and hence we have
\[
  \mathfrak{C}_1' \ll M^{3/2}
  \quad \textrm{and} \quad
  \mathfrak{C}_1'' \ll M^{5/2}.
\]
Thus we get
\begin{align}\label{eqn:char_sum02}
  \mathfrak{C}(0)
  & \ll M^{5/2}
  \sum_{d|q_1q_2}\sum_{d'|q_1q_2}dd'
  \frac{rq_1q_2}{[d,d']}  \delta_{(d,d')\mid (m\ell'-m'\ell)}.
\end{align}
As in \eqref{eqn:sum-l&m}, by Lemma \ref{lemma:J}, we have
\begin{align*} 
  \mathop{\sum_{\ell}\sum_{\ell'}\sum_{m}\sum_{m'}}_{(d,d')\mid (m\ell'-m'\ell)} |\mathfrak{J}(0)|
  &
  \ll N^\varepsilon  L M_1 \left(\frac{ LM_1 (\frac{LNX}{MRQt})^2 +LM_1\frac{MRQ}{LNX}  }{(d,d')} +1 \right) \frac{1}{t} .
\end{align*}
Hence, similar to the estimate for $\Omega_{01}$, we have
\begin{align*}
  \Omega_{02}&
  \ll N^\varepsilon \frac{r^2 N^{3/2} L^{11/3} R q_1^2  M^{7/2} X^3}{n_1^2 Q^3}
  \bigg(\frac{L^3 N  X^2}{R Q^2} +  \frac{M^3R^2Q t^2}{ N^2X } +  1  \bigg).
\end{align*}
Hence, similar to the estimate for \eqref{eqn:Omega01<<}, the contribution from $\Omega_{02}$ to \eqref{eqn:led to estimate} is
\begin{align} \label{eqn:Omega02<<}
    &\ll  N^{\varepsilon} r^{1/2}  N^{1/2} L^{1/2} M^{5/4} K^{3/2}
   + N^{\varepsilon} r^{1/2}  N^{1/2} L^{1/2} M^{5/4} t +
   N^{\varepsilon} r^{1/2} \frac{N^{3/4}  M^{1/2}K^{3/4} }
    { L^{1/4}}.
\end{align}

\subsection{Case a: $\frac{LNX}{MRQ} \gg t^{1-\varepsilon}$}

By the same argument as in the \S \ref{subsec:zero-freq-generic} and Lemma \ref{lemma:J2} we have
\begin{align*}
  \Omega_{0}&
  \ll N^\varepsilon \frac{N_0 M^3}{n_1^2 M_1^{1/2}}
  \frac{L^{2/3} q_1^3}{R^3}
  \sum_{\substack{q_2\sim R/q_1\\ (q_2,rn_1)=1}}
  \sum_{\substack{d|q \\d'|q}}
   rq   LM_1
   \left(\frac{LM_1}{M} +  q \right) \frac{MRQ}{NLX}
  \\
  &  \hskip  30pt +
  N^\varepsilon \frac{N_0 M^{5/2}}{n_1^2 M_1^{1/2}}
  \frac{L^{2/3} q_1^3}{R^3}
  \sum_{\substack{q_2\sim R/q_1\\ (q_2,rn_1)=1}}
  \sum_{\substack{d|q \\d'|q}}
   rq   LM_1
   \left( LM_1   +  q \right)  \frac{MRQ}{NLX} \\
  &
  \ll N^\varepsilon \frac{N_0}{n_1^2 M_1^{1/2}} \frac{L^{2/3} q_1^3}{R^3}
  \sum_{\substack{q_2\sim R/q_1\\ (q_2,rn_1)=1}}
   rq   LM_1
   \left(LM_1 M^{5/2} +  q M^3 \right) \frac{MRQ}{NLX}
   \\
   &
  \ll N^\varepsilon
  \frac{r^2 N^{3/2} M q_1^2 L^{11/3} X^3   }{n_1^2 Q^3}
   \left(\frac{NL^3 M^{5/2} X^2}{Q^2}  +  R M^3 \right) .
\end{align*}
Here we have used $N_0=\frac{N^2L^2X^3r}{Q^3}$ and $M_1\ll \frac{NL^2X^2}{Q^2} N^\varepsilon$.
Therefore, the contribution from $\Omega_{0}$ to \eqref{eqn:led to estimate} is
\begin{align*} 
  &\ll \frac{N^{3/4+\varepsilon}X^{1/2}}{M^{2}L^{4/3}Q^{1/2}r^{1/2}}
   \sum_{n_1\ll RMr} n_1^{\theta_3}\sum_{\frac{n_1}{(n_1,r)}|q_1|(rn_1)^\infty} \frac{1}{n_1 q_1^{1/2}}  \nonumber
   \\
   & \hskip 60pt \cdot
  \left( \frac{r^2 N^{3/2} M   L^{11/3} X^3   }{ Q^3}
   \left(\frac{NL^3 M^{5/2} X^2}{Q^2}  +  R M^3 \right)  \right)^{1/2}
  \nonumber
  \\
    &\ll N^{\varepsilon} r^{1/2} \frac{ N^{3/2}  L^{1/2}   X^{2} }
     {M^{3/2}  Q^{2}}
   \left(\frac{N^{1/2}  L^{3/2} M^{5/4} X}{Q}  +  R^{1/2}  M^{3/2} \right).
\end{align*}
Note that we have $R\ll \frac{NLX}{MQt^{1-\varepsilon}}$ now.
By this and inserting $Q=(\frac{NL}{MK})^{1/2}$, one can bounded the above by
\begin{align} \label{eqn:Omega03<<}
    &\ll N^{\varepsilon} r^{1/2}  N^{1/2}  L^{1/2} M^{5/4}   K^{3/2}
     +
   N^{\varepsilon} r^{1/2} \frac{N^{3/4}   M^{3/4} K^{5/4}}
     { L^{1/4}   t^{1/2}}.
\end{align}

\subsection{Case b: $\frac{LNX}{MRQ} \asymp  t$}

By the same argument as in the \S \ref{subsec:zero-freq-generic} and Lemma \ref{lemma:J3} we have
\begin{align*}
  \Omega_{0}&
  \ll N^\varepsilon \frac{N_0 M^3}{n_1^2 M_1^{1/2}}
  \frac{L^{2/3} q_1^4}{R^4}
  \sum_{\substack{q_2\sim R/q_1\\ (q_2,rn_1)=1}}
  \sum_{\substack{d|q \\d'|q}}
   rq   LM_1
   \left(\frac{LM_1}{M} +  q \right) \frac{MRQ}{NLX}
  \\
  &  \hskip  30pt +
  N^\varepsilon \frac{N_0 M^{5/2}}{n_1^2 M_1^{1/2}}
  \frac{L^{2/3} q_1^4}{R^4}
  \sum_{\substack{q_2\sim R/q_1\\ (q_2,rn_1)=1}}
  \sum_{\substack{d|q \\d'|q}}
   rq   LM_1
   \left( LM_1   +  q \right)  \frac{MRQ}{NLX} \\
  &
  \ll N^\varepsilon \frac{N_0}{n_1^2} \frac{L^{2/3} q_1^3}{R^3}
   r R L M_1^{1/2}
   \left(LM_1 M^{5/2} +  R M^3 \right) \frac{MRQ}{NLX}.
\end{align*}
By $N_0=\frac{N^2L^2X^3r}{Q^3}$ and $M_1\ll \frac{M^2R^2}{N} t^\varepsilon$ we obtain
\begin{align*}
  \Omega_{0}&
  \ll N^\varepsilon \frac{1}{n_1^2}
  \frac{N^2L^2X^3r}{Q^3} \frac{MRQ}{NLX} \frac{L^{2/3} q_1^3}{R^3}
   r R L \frac{MR}{N^{1/2}}
   \left(L  M^{5/2} \frac{M^2R^2}{N}  +  R M^3 \right)
   \\
  & \ll N^\varepsilon
  \frac{r^2 N^{1/2} M^2 q_1^3 L^{8/3} X^2 }{n_1^2  Q^2}
   \left(\frac{LM^{9/2}R^2}{N}  +  R M^3 \right) .
\end{align*}
Thus, the contribution from $\Omega_{0}$ to \eqref{eqn:led to estimate case b} is
\begin{align*} 
  &\ll \frac{N^{1+\varepsilon}X^{1/2}}{M^{5/2}L^{4/3}Q^{1/2}r^{1/2}}
   \sum_{n_1\ll RMr} n_1^{\theta_3}\sum_{\frac{n_1}{(n_1,r)}|q_1|(rn_1)^\infty} \frac{1}{n_1 q_1^{1/2}}  \nonumber
   \\
   & \hskip 60pt \cdot
  \left( \frac{r^2 N^{1/2} M^2  L^{8/3} X^2 }{ Q^2}
   \left(\frac{LM^{9/2}R^2}{N}  +  R M^3 \right) \right)^{1/2}
  \nonumber
  \\
    &\ll N^{\varepsilon} r^{1/2}  \frac{N^{3/4}  L^{1/2}M^{3/4}R X^{3/2} }{ Q^{3/2}}   +
     N^{\varepsilon} r^{1/2} \frac{ N^{5/4}  R^{1/2}  X^{3/2} } { Q^{3/2}} .
\end{align*}
By $Q=(\frac{NL}{MK})^{1/2}$ again and noting that $R\asymp  \frac{NLX}{MQt}$,
we deduce that the above is dominated by
\begin{align} \label{eqn:Omega0b<<}
    &\ll N^{\varepsilon} r^{1/2}  N^{1/2}  L^{1/4} M   \frac{ K^{5/4} }{t}
     +
   N^{\varepsilon} r^{1/2} \frac{N^{3/4}   M^{1/2}  K}
     { L^{1/2}   t^{1/2}}.
\end{align}

\subsection{Case c: $\frac{LNX}{MRQ} \ll t^{\varepsilon}$}

By the same argument as in the \S \ref{subsec:zero-freq-generic} and \eqref{eqn:I-c} we have (taking $M_1\asymp \frac{R^2M^2t^2}{N}$)
\begin{align*}
  \Omega_{0}&
  \ll N^\varepsilon \frac{N_0 M^3}{n_1^2 M_1^{1/2}}
  \frac{L  q_1^4}{R^4}
  \sum_{\substack{q_2\sim R/q_1\\ (q_2,rn_1)=1}}
  \sum_{\substack{d|q \\d'|q}}
   rq   LM_1
   \left(\frac{LM_1}{M} +  q \right)
  \\
  &  \hskip  30pt +
  N^\varepsilon \frac{N_0 M^{5/2}}{n_1^2 M_1^{1/2}}
  \frac{L  q_1^4}{R^4}
  \sum_{\substack{q_2\sim R/q_1\\ (q_2,rn_1)=1}}
  \sum_{\substack{d|q \\d'|q}}
   rq   LM_1
   \left( LM_1   +  q \right)   \\
  &
  \ll N^\varepsilon \frac{N_0}{n_1^2} \frac{L  q_1^3}{R^3}
   r R  L M_1^{1/2}
   \left(LM_1 M^{5/2} +  R M^3 \right)  .
\end{align*}
By $N_0 \ll \frac{R^3 M^3 r}{LN}t^\varepsilon$ and $M_1\asymp \frac{R^2M^2t^2}{N}$, one has
\begin{align*}
   \Omega_{0} &
  \ll N^\varepsilon \frac{1}{n_1^2}
  \frac{R^3 M^3 r}{LN}   \frac{L  q_1^3}{R^3}
   r R  L \frac{RMt}{N^{1/2}}
   \left(L  M^{5/2} \frac{R^2M^2t^2}{N}  +  R M^3 \right)
   \\
   &
  \ll N^\varepsilon
  \frac{r^2 q_1^3 R^2 L  M^4 t}{n_1^2 N^{3/2}}
   \left(\frac{LM^{9/2}R^2 t^2}{N}  +  R M^3 \right) .
\end{align*}
So the contribution from $\Omega_{0}$ to \eqref{eqn:led to estimate c} is
\begin{align*} 
  &\ll \frac{N^{5/4+\varepsilon}X }{M^{5/2}L Q r^{1/2} t^{1/2}}
  \left(\frac{r^2R^2 L  M^4 t}{N^{3/2}}
   \left(\frac{LM^{9/2}R^2 t^2}{N}  +  R M^3 \right)  \right)^{1/2}
  \nonumber
     \\
    &\ll N^{\varepsilon} r^{1/2} \frac{ M^{7/4}R^2 t  X} {  Q}
      + N^{\varepsilon} r^{1/2} \frac{ N^{1/2}M  R^{3/2} X}
     {L^{1/2}   Q  }.
\end{align*}
Now we have the condition $X\ll \frac{MRQ}{LN}t^\varepsilon$,
so one computes the above as
\begin{align} \label{eqn:Omega0c<<}
   &\ll N^{\varepsilon} r^{1/2} N^{1/2} L^{1/2} M^{5/4} \frac{ t  } {  K^{3/2}}
      + N^{\varepsilon} r^{1/2} \frac{ N^{3/4} M^{3/4} }  {L^{1/4}     K^{5/4}  }.
\end{align}

Combining \eqref{eqn:Omega01<<}, \eqref{eqn:Omega02<<}, \eqref{eqn:Omega03<<}, \eqref{eqn:Omega0b<<}
and \eqref{eqn:Omega0c<<}, we see that the contribution of the zero frequency is dominated by
\begin{multline}\label{eqn:contribution zero frequency}
  \ll N^{\varepsilon} r^{1/2}  N^{1/2} L^{1/2} M^{5/4} K^{3/2}
   + N^{\varepsilon} r^{1/2}  N^{1/2} L^{1/2} M^{5/4} t
   \\
   +
   N^{\varepsilon} r^{1/2} \frac{N^{3/4}  M^{3/4}K^{3/4} }
    { L^{1/4}}
    +
   N^{\varepsilon} r^{1/2} \frac{N^{3/4}   M^{3/4} K^{5/4}}
     { L^{1/4}   t^{1/2}}.
\end{multline}


\section{The non-zero frequencies}\label{sec:non-zero-freq}

\subsection{$n_2\not\equiv 0\;(\mod M)$}
Denote the contribution from $n_2\not\equiv0 \; (\bmod M)$
in $\Omega_\mathbf{*}$ by $\Omega_{\mathbf{*},1}$,
where $\mathbf{*}\in\{\mathbf{a},\mathbf{b},\mathbf{c}\}$.
We have
\begin{align*}
  \mathfrak{C}(n_2)\ll |\mathfrak{C}_1(n_2)\mathfrak{C}_2(n_2)\mathfrak{C}_3(n_2)|,
\end{align*}
where
\begin{align*}
  \mathfrak{C}_1(n_2)& = \;\sideset{}{^\star}\sum_{b\bmod M }
   \;\sideset{}{^\star}\sum_{b'\bmod M }\left(\sum_{\substack{u\bmod M \\ u\neq b}}\bar{\chi}(u)e\left(\frac{m\overline{q_2^2\ell(b-u)}}{M}\right)\right)
  \\ &
  \cdot
  \left(\sum_{\substack{u'\bmod M \\ u'\neq b'}}\chi(u')e\left(-\frac{m'\overline{q_2'^2\ell(b'-u')}}{M}\right)\right)
  \left(\sideset{}{^\star}\sum_{\substack{\alpha,\alpha'\bmod M \\q_2'\bar{\alpha}-q_2\bar{\alpha'}\equiv-\eta_1n_2(M)}}
  e\left(\frac{\alpha\overline{bq_2^2}-\alpha'\overline{b'q_2'^2}}{M}\right)\right),
\end{align*}
\begin{align*}
  \mathfrak{C}_2(n_2)=\sum_{d_1|q_1}\sum_{d_1'|q_1}d_1d_1'\underset{q_2'\alpha_1-q_2\bar{\alpha_1'}\equiv -\eta_1n_2(\bmod \frac{rq_1}{n_1})}{\sideset{}{^\star}\sum_{\substack{\alpha_1 \bmod \frac{rq_1}{n_1} \\
  n_1\alpha_1\equiv -m\bar{\ell} \bmod d_1 }}
  \;
  \sideset{}{^\star}\sum_{\substack{\alpha_1' \bmod \frac{rq_1}{n_1} \\
  n_1\alpha_1'\equiv -m'\bar{\ell'} \bmod d_1' }}}1,
\end{align*}
and
\begin{align*}
  \mathfrak{C}_3(n_2)=\sum_{d_2|q_2}\sum_{d_2'|q_2'}d_2d_2'
  \sideset{}{^\star}\sum_{\substack{\alpha_2 (q_2), \; \alpha_2' (q_2')\\
  q_2'\bar{\alpha_2}-q_2\bar{\alpha_2'}\equiv-\eta_1n_2 \bmod q_2q_2'\\
  n_1\alpha_2\equiv -m\bar{\ell} \bmod d_2 \\
  n_1\alpha_2'\equiv -m'\bar{\ell'} \bmod d_2' }}1.
\end{align*}
For $\mathfrak{C}_2(n_2)$, the congruence condition determines at most one solution of $\alpha_1'$  in terms of $\alpha_1$,
  and hence we have
  \begin{equation*}
    \mathfrak{C}_2(n_2) \leq  \sum_{d_1\mid q_1} d_1
  \sum_{d_1'\mid q_1} d_1'
    \sideset{}{^\star}\sum_{\substack{\alpha_1 \bmod rq_1/n_1 \\ - m \bar\ell \equiv n_1\alpha_1   \bmod d_1 }}  1
    .
  \end{equation*}
  Note that $\alpha_1$ is uniquely determined modulo $d_1/(d_1,n_1)$. Since $(\frac{d_1}{(d_1,n_1)},\frac{n_1}{(d_1,n_1)})=1$,  $\frac{d_1}{(d_1,n_1)}\mid \frac{q_1}{(d_1,n_1)}$ and $\frac{n_1}{(d_1,n_1)}\mid \frac{rq_1}{(d_1,n_1)}$, we have
  $\frac{d_1}{(d_1,n_1)}\mid \frac{rq_1}{n_1}$. Hence we have
  \[
    \mathfrak{C}_2(n_2)  \ll \frac{rq_1}{n_1} \sum_{d_1\mid q_1}
  \sum_{d_1'\mid q_1} d_1' (d_1,n_1) \delta_{(d_1,n_1)\mid m }.
  \]
  Similarly by considering $\alpha_1$-sum first we have
  \[
    \mathfrak{C}_2(n_2)  \ll \frac{rq_1}{n_1} \sum_{d_1\mid q_1}
  \sum_{d_1'\mid q_1} d_1 (d_1',n_1) \delta_{(d_1',n_1)\mid m' }.
  \]
For $\mathfrak{C}_3(n_2)$, from the congruence $q_2'\bar{\alpha}-q_2\bar{\alpha'}\equiv-\eta_1n_2 \bmod q_2q_2'$ we have $(q_2,q_2')\mid n$.
Since $(n_1,q_2)=1$, we have $\alpha \equiv  - m\bar\ell \bar{n}_1 \bmod d_2$ and hence $ q_2' \bar\alpha \equiv - \eta_1 n_2 \bmod d_2$. Therefore we get $ n_1 q_2' \equiv \eta_1 mn_2 \bar\ell  \bmod d_2$. Similarly we have $- n_1 q_2 \equiv \eta_1 m'n_2 \bar{\ell'}\bmod d_2'$. Note that the congruence determines $\alpha_2 \bmod [q_2/(q_2,q_2'),d_2]$ and for each given $\alpha_2$ we have at most one solution of $\alpha_2' \bmod q_2'$. Hence we have
  \[
    \mathfrak{C}_3(n) \ll \mathop{\sum\sum}_{\substack{d_2 \mid (q_2, -q_2' n_1\ell + \eta_1mn_2) \\
    d_2' \mid (q_2', q_2n_1\ell' + \eta_1m'n_2)}} d_2d_2' \; \frac{q_2}{[q_2/(q_2,q_2'),d_2]} \; \delta_{(q_2,q_2')\mid n}.
  \]
Similarly we have
\[
  \mathfrak{C}_3(n_2)\ll \underset{\substack{d_2|(q_2,-q_2'n_1\ell+\eta_1mn_2)
  \\d_2'|(q_2',q_2n_1\ell'+\eta_1m'n_2)}}{\sum\sum}d_2d_2'
   \frac{q_2'}{[q_2'/(q_2,q_2'),d_2']}
  \delta_{(q_2,q_2')|n_2}.
\]
Together with \cite[Eq. (5.6)]{Sharma2019} and \cite[Proposition 4.4]{LMS}, we have
\begin{align*}
  &\mathfrak{C}_1(n_2)\ll M^{5/2},
  \\
  &\mathfrak{C}_2(n_2)\ll \frac{q_1r}{n_1}
  \sum_{d_1|q_1}\sum_{d_1'|q_1}
  \min\{d_1'(d_1,n_1)\delta_{(d_1,n_1)|m},d_1(d_1',n_1)\delta_{(d_1',n_1)|m'}\},
  \\
  &\mathfrak{C}_3(n_2)\ll \underset{\substack{d_2|(q_2,-q_2'n_1\ell+\eta_1mn_2)
  \\d_2'|(q_2',q_2n_1\ell'+\eta_1m'n_2)}}{\sum\sum}d_2d_2'
  \min\left\{\frac{q_2}{[q_2/(q_2,q_2'),d_2]},\frac{q_2'}{[q_2'/(q_2,q_2'),d_2']}\right\}
  \delta_{(q_2,q_2')|n_2}.
\end{align*}

Now, we need some  careful counting to estimate $\Omega_{\mathbf{*},1}$ (cf. \cite[\S6]{Munshi2018}, \cite[\S5]{Sharma2019}, \cite[\S 6]{LMS} and \cite[\S 4.5]{LinSun}).
\subsubsection{Case a}\label{subsub: n2 not equiv 0 and case a}
It is obvious that, for fixed tuple $(n_1,\alpha,n_2)$, the congruence
\begin{align*}
        -q_2'n_1\ell+\eta_1mn_2 \equiv 0(\bmod d_2)
\end{align*}
has a solution if and only if $(d_2,n_2)|q_2'\ell$, in which case $m$ is uniquely determined modulo
$\frac{d_2}{(d_2,n_2)}$.
Combining this together with the condition $\delta_{(d_1,n_1)|m}$ in $\mathfrak{C}_2(n_2)$,
the number of $m$ ($\sim M_1$) is dominated by
$\delta_{(d_2,n_2)|q_2'}O(1+\frac{M_1(d_2,n_2)}{(d_1,n_1)d_2})$.
Then, we get
\begin{align*}
  \Omega_{\mathbf{a},1}\ll & \frac{q_1^4rN_0L^{2/3}M^{5/2}}{n_1^3M_1^{1/2}R^3}\sum_{d_1|q_1}\sum_{d_1'|q_1}
  d_1'(d_1,n_1)\underset{(\ell\ell',q_1)=1}
  {\sum_{\ell\in\mathcal{L}}\sum_{\ell'\in\mathcal{L}}}|A_\pi(1,\ell)A_{\pi}(1,\ell')|
  \\
  & \cdot
  \underset{(q_2q_2',n_1M)=1}{\sum_{\substack{q_2\sim R/q_1\\(q_2,\ell)=1}}\sum_{\substack{q_2'\sim R/q_1\\(q_2',\ell')=1}}}
  \sum_{d_2|q_2}\sum_{d_2'|q_2'}
  \sum_{\substack{1\leq n_2\leq N_2\\ (d_2,n_2)|q_2'\ell\\ (q_2,q_2')|n_2}}d_2d_2'
  \left(1+\frac{M_1(d_2,n_2)}{(d_1,n_1)d_2}\right)
  \\
  & \cdot \min\left\{\frac{q_2}{[q_2/(q_2,q_2'),d_2]},\frac{q_2'}{[q_2'/(q_2,q_2'),d_2']}\right\}
  \sum_{\substack{m'\sim M_1\\ q_2n_1\ell'+\eta_1m'n_2\equiv 0(\bmod d_2')}}|\lambda_f(m')|^2
  |\mathfrak{J}_{\mathbf{a}}(n_2)|.
\end{align*}
Let us make the following notation:
\begin{align*}
  & (q_2,q_2')=q_3,\quad q_2=q_3q_4,\quad q_2'=q_3q_4'
  \\
  & d_2=d_0d_3d_4, \quad d_0|(q_3,q_4),\quad d_3|q_3, \quad (d_3,q_4)=1,\quad (d_4,q_3)=1,\quad d_4|q_4,
  \\
  &d_2'=d_3'd_4', \quad d_3'|q_3, \quad d_4'|q_4.
\end{align*}
It is easy to see that
$(d_2,n_2)\leq (d_0d_3,n_2)(d_4,n_2)\leq d_0d_3(d_4,n_2)=d_0d_3(d_4,\frac{n_2}{q_3})$,
$\frac{q_2}{[q_2/(q_2,q_2'),d_2]}=\frac{q_3q_4}{[q_4,d_2]}\leq \frac{q_3}{d_3}$,
and
$\frac{q_2'}{[q_2'/(q_2,q_2'),d_2']}=\frac{q_3q_4'}{[q_4',d_2']}\leq \frac{q_3q_4'}{d_2'}$.
Then, breaking the $n_2$-sum into dyadic segments $n_2\sim \tilde{N}_2$ with $\tilde{N}_2\ll N_2$
and using Lemma \ref{lemma:J} \& Lemma \ref{lemma:J2}, one has
\begin{align*}
\begin{split}
  \Omega_{\mathbf{a},1}\ll & \sup_{\substack{1\ll\tilde{N}_2\ll N_2\\ \textup{dyadic}}}
  \frac{N^\varepsilon q_1^4rN_0L^{2/3}M^{5/2}}{n_1^3M_1^{1/2}R^3}\sum_{d_1|q_1}\sum_{d_1'|q_1}(d_1,n_1)d_1'
  \underset{(\ell\ell',q_1)=1}{\sum_{\ell\in\mathcal{L}}\sum_{\ell'\in\mathcal{L}}}
  |A_\pi(1,\ell)A_{\pi}(1,\ell')|
  \\
  & \cdot
  \sum_{\substack{q_3\leq R/q_1\\(q_3,n_1\ell\ell'M)=1}}
  \underset{(q_4q_4',n_1M)=1}{\sum_{\substack{q_4\sim R/q_3q_1\\(q_4,\ell)=1}}\sum_{\substack{q_4'\sim R/q_3q_1\\(q_4',\ell')=1}}}
  \sum_{d_0|(q_3,q_4)}\sum_{\substack{d_3|q_3\\(d_3,q_4)=1}}
  \sum_{\substack{d_4|q_4\\(d_4,q_3)=1}}d_0d_3d_4
  \\
  & \cdot \sum_{d_3'|q_3}\sum_{d_4'|q_4'}d_3'd_4'
  \sum_{\substack{n_2\sim \tilde{N}_2\\ (d_2,n_2)|q_3q_4'\ell\\ q_3|n_2}}
  \left(1+\frac{M_1(d_4,n_2/q_3)}{(d_1,n_1)d_4}\right)C(\tilde{N}_2)
  \\
  & \cdot \min\left\{\frac{q_3}{d_3},\frac{q_3q_4'}{d_3'd_4'}\right\}
  \sum_{\substack{m'\sim M_1\\ q_3q_4n_1\ell'+\eta_2m'n_2\equiv 0(\bmod d_3'd_4')}}|\lambda_f(m')|^2,
\end{split}
\end{align*}
where
\begin{align}\label{C(N2)}
C(\tilde{N}_2)=
  \begin{cases}
        \frac{RQ^{3/2}M^{1/2}n_1^{1/2}}{tNLX^{3/2}q_1^{1/2}\tilde{N}_2^{1/2}}\quad
        & N_2'\ll \tilde{N}_2\ll N_2\, \& \,\frac{NLX}{MRQ}\ll t^{1-\varepsilon}
        \\
        \frac{1}{t} \quad &\tilde{N}_2\ll N_2'\, \& \, \frac{NLX}{MRQ}\ll t^{1-\varepsilon}
        \\
        \frac{MRQ}{NLX} &\tilde{N}_2\ll N_2\, \& \, \frac{NLX}{MRQ}\gg t^{1-\varepsilon}
     \end{cases}
\end{align}

$\mathbf{Case\ (i):}$ $\bm{q_3q_4n_1\ell'+\eta_2m'n_2\neq 0}$.
Denote the contribution from this part in $\Omega_{\mathbf{a},1}$ by $\Omega_{\mathbf{a},11}$.
Write $q_3=d_3'q_5$, $q_4=d_0q_6$, and $q_4'=d_4'q_6'$, then we have
\begin{align*}
  \Omega_{\mathbf{a},11}\ll & \sup_{\substack{1\ll\tilde{N}_2\ll N_2\\ \textup{dyadic}}}\frac{N^\varepsilon q_1^5rN_0L^{2/3}M^{5/2}}{n_1^3M_1^{1/2}R^3}\sum_{d_1|q_1}(d_1,n_1)
  \sum_{\ell\in\mathcal{L}}\sum_{\ell'\in\mathcal{L}}|A_\pi(1,\ell)A_{\pi}(1,\ell')|
  \\
  & \cdot
  \sum_{d_3'\leq R/q_1}d_3'\sum_{q_5\leq R/q_1d_3'}\sum_{\substack{d_0\leq R/q_1d_3'q_5\\ d_0|d_3'q_5}}d_0\sum_{\substack{d_3|d_3'q_5\\(d_3,q_4)=1}}\sum_{q_6\sim R/q_1d_3'd_0q_5}
  \sum_{\substack{d_4|d_0q_6\\(d_4,d_3'q_5)=1}}
  \\
  & \cdot \sum_{\substack{n_2\sim \tilde{N}_2\\ d_3'q_5|n_2}}d_3'q_5
  \left(d_4+\frac{M_1(d_4,n_2/d_3'q_5)}{(d_1,n_1)}\right)C(\tilde{N}_2)\sum_{m'\sim M_1}|\lambda_f(m')|^2
  \\
  & \cdot \sum_{\substack{d_4'\leq R/q_1d_3'q_5\\ 0\neq d_3'd_4q_5q_6n_1\ell'+\eta_1m'n_2\equiv 0(\bmod d_3'd_4')}}d_4'\sum_{q_6'\sim R/q_1d_3'q_5d_4'}1,
\end{align*}
By the well known bound of the divisor function, the number of the tuple $(d_0,d_3,d_4,d_4')$ is bounded by $O(N^\varepsilon)$. Combining this together with \eqref{eqn:RS2} and \eqref{eqn:RS3}, we get
\begin{align}\label{Omega1}
  \Omega_{\mathbf{a},11}\ll \sup_{\substack{1\ll\tilde{N}_2\ll N_2\\ \textup{dyadic}}}
  \frac{N^\varepsilon q_1^3rN_0M^{5/2}M_1^{1/2}L^{8/3}\tilde{N}_2C(\tilde{N}_2)}{n_1^3R}
  (R+M_1)
\end{align}

$\mathbf{Case\ (ii):}$ $\bm{q_3q_4n_1\ell'+\eta_2m'n_2=0}$.
Denote the contribution from this part in $\Omega_{\mathbf{a},1}$ by $\Omega_{\mathbf{a},12}$.
In this subsection, we use $(d_2,n_2)\leq (q_2'\ell, q_2)=q_3$.
It is therefore we have
\begin{align*}
  \Omega_{\mathbf{a},12}\ll & \sup_{\substack{1\ll\tilde{N}_2\ll N_2\\ \textup{dyadic}}}\frac{N^\varepsilon q_1^4rN_0L^{2/3}M^{5/2}}{n_1^3M_1^{1/2}R^3}\sum_{d_1|q_1}\sum_{d_1'|q_1}d_1'(d_1,n_1)
  \\
  & \cdot \sum_{\ell\in\mathcal{L}}\sum_{\ell'\in\mathcal{L}}|A_\pi(1,\ell)A_{\pi}(1,\ell')|
  \sum_{q_3\leq R/q_1}\sum_{q_4,q_4'\sim R/q_1q_3}\sum_{d_0|(q_3,q_4)}\sum_{\substack{d_3|q_3\\(d_3,q_4)=1}}
  \sum_{\substack{d_4|q_4\\(d_4,q_3)=1}}d_0d_3d_4
  \\
  & \cdot \sum_{d_3'|q_3}\sum_{d_4'|q_4'}d_3'd_4'
  \sum_{\substack{n_2\sim \tilde{N}_2\\ q_3|n_2}}
  \left(1+\frac{M_1q_3}{(d_1,n_1)d_0d_3d_4}\right)
  C(\tilde{N}_2)
  \\
  & \cdot \min\left\{\frac{q_3}{d_3},\frac{q_3q_4'}{d_3'd_4'}\right\}
  \sum_{\substack{m'\sim M_1\\ q_3q_4n_1\ell'+\eta_1m'n_2=0}}|\lambda_f(m')|^2
  \\
  \ll & \sup_{\substack{1\ll\tilde{N}_2\ll N_2\\ \textup{dyadic}}}\frac{N^\varepsilon q_1^5rN_0L^{2/3}M^{5/2}}{n_1^3M_1^{1/2}R^3}\sum_{d_1|q_1}(d_1,n_1)
  \sum_{m'\sim M_1}|\lambda_f(m')|^2\sum_{q_3\leq R/q_1}q_3
  \sum_{\substack{n_2\sim \tilde{N}_2\\ q_3|n_2}}C(\tilde{N}_2)
  \\
  & \cdot \sum_{q_4\sim R/q_1q_3} \sum_{\ell\in\mathcal{L}}|A_{\pi}(1,\ell)|
  \left(\frac{R}{q_1}+\frac{M_1q_3}{(d_1,n_1)}\right)
  \\
  & \cdot
  \sum_{d_0|(q_3,q_4)}\sum_{\substack{d_3|q_3\\(d_3,q_4)=1}}
  \sum_{\substack{d_4|q_4\\(d_4,q_3)=1}}\sum_{d_3'|q_3}\sum_{\ell'\in\mathcal{L}}|A_\pi(1,\ell')|
  \delta_{q_3q_4\ell'|m'n_2}
  \sum_{q_4'\sim R/q_1q_3}\sum_{d_4'|q_4'}q_4'.
\end{align*}
Now, we estimate the last two sums trivially, and then use the condition $\delta_{q_3q_4\ell'|m'n_2}$
together with \eqref{eqn:RS3} and \eqref{eqn:Bound A(r,n)}, obtaining
\begin{align*}
  &\sum_{q_4\sim R/q_1q_3} \sum_{\ell\in\mathcal{L}}|A_{\pi}(1,\ell)|
  \left(\frac{R}{q_1}+\frac{M_1q_3}{(d_1,n_1)}\right)
  \\
  & \cdot
  \sum_{d_0|(q_3,q_4)}\sum_{\substack{d_3|q_3\\(d_3,q_4)=1}}
  \sum_{\substack{d_4|q_4\\(d_4,q_3)=1}}\sum_{d_3'|q_3}\sum_{\ell'\in\mathcal{L}}|A_\pi(1,\ell')|
  \delta_{q_3q_4\ell'|m'n_2}
  \sum_{q_4'\sim R/q_1q_3}\sum_{d_4'|q_4'}q_4'
  \\
  & \ll \frac{R^2L^{1+\theta_3}}{q_1^2q_3^2}\left(\frac{R}{q_1}+\frac{M_1q_3}{(d_1,n_1)}\right),
\end{align*}
where $\theta_3\leq \frac{5}{14}$.
Therefore, it follows that
\begin{align}\label{Omega2}
\begin{split}
  \Omega_{\mathbf{a},12} \ll & \sup_{\substack{1\ll\tilde{N}_2\ll N_2\\ \textup{dyadic}}}\frac{N^\varepsilon q_1^4rN_0L^{2/3}M^{5/2}}{n_1^3M_1^{1/2}R^3}\sum_{d_1|q_1}(d_1,n_1)
  \sum_{m\sim M_1}|\lambda_f(m')|^2\sum_{q_3\leq R/q_1}q_3
  \sum_{\substack{n_2\sim \tilde{N}_2\\ q_3|n_2}}C(\tilde{N}_2)
  \\
  & \cdot \frac{R^2L^{1+\theta_3}}{q_1^2q_3^2}\left(\frac{R}{q_1}+\frac{M_1q_3}{(d_1,n_1)}\right)
  \\
  \ll & \sup_{\substack{1\ll\tilde{N}_2\ll N_2\\ \textup{dyadic}}}
  \frac{q_1^3rN^\varepsilon N_0M^{5/2}M_1^{1/2}L^{5/3+\theta_3}
  \tilde{N}_2C_1(\tilde{N}_2)}{n_1^3R}
  (R+M_1).
\end{split}
\end{align}
Recall
\begin{align}\label{parameters}
  \begin{split}
     & Q=\left(\frac{NL}{MK}\right)^{1/2},\quad N_0=\frac{N^2L^2X^3r}{Q^3}, \quad N_2=\frac{Q^2Rn_1}{NLX^2q_1}t^\varepsilon, \\
     & N_2'=t^\varepsilon\left(\frac{NLn_1}{M^2Rt^2q_1}+\frac{R^2Q^3Mn_1}{N^2L^2X^3q_1}\right),\quad
     N\ll \frac{(Mt)^{3+\varepsilon}}{r^2}.
  \end{split}
\end{align}
For $\frac{NLX}{MRQ}\ll t^{1-\varepsilon}$, we have $M_1\ll \frac{t^2R^2M^2}{N}$.
By taking $L=M^{1/4}$ and $K=t^{4/5}$, one has
$R+M_1\ll \frac{t^2M^2RQ}{N}$.
Hence, by applying these bounds into \eqref{Omega1} and \eqref{Omega2}, we derive that
\begin{align*}
  \Omega_{\mathbf{a},1}\ll \sup_{\substack{1\ll\tilde{N}_2\ll N_2\\ \textup{dyadic}}}
  \frac{N^\varepsilon q_1^3rN_0M^{11/2}RQt^3L^{8/3}\tilde{N}_2C(\tilde{N}_2)}{n_1^3N^{3/2}}.
\end{align*}
Combining this together with \eqref{C(N2)} and \eqref{parameters}, we get
\begin{align*}
  \Omega_{\mathbf{a},1}\ll \frac{N^\varepsilon q_1^2r^2}{n_1^2}
  \left(\frac{Q^3L^{19/6}M^6t^2}{N} +
  \frac{N^{3/2}L^{17/3}M^{7/2}}{Q^2} + \frac{Q^4L^{8/3}M^{13/2}t^2}{N^{3/2}}\right).
\end{align*}
For $\frac{NLX}{MRQ}\gg t^{1-\varepsilon}$, we have $M_1\ll\frac{NL^2X^2}{Q^2}$
and $R\ll \frac{NLX}{MQt^{1-\varepsilon}}$. Thus, in this case, we arrive at
\begin{align*}
  \Omega_{\mathbf{a},1}\ll \frac{N^\varepsilon q_1^2r^2}{n_1^2}
  \left(\frac{N^{5/2}L^{17/3}M^{3/2}}{Q^3t^2} +
  \frac{N^{5/2}L^{20/3}M^{5/2}}{Q^{4}t}\right).
\end{align*}

Therefore, the contribution from $\Omega_{\mathbf{a},1}$ to
\eqref{eqn:led to estimate} is
\begin{multline}\label{eqn:contribiton case a}
  \ll \frac{N^{3/4+\varepsilon}X^{1/2}}{M^{2}L^{4/3} Q^{1/2}r^{1/2}}
   \sum_{n_1\ll RMr} n_1^{\theta_3}\sum_{\frac{n_1}{(n_1,r)}|q_1|(rn_1)^\infty}
   \frac{r}{n_1q_1^{1/2}}\left(\frac{Q^3L^{19/6}M^6t^2}{N} \right.
  \\
  \left.
  + \frac{N^{3/2}L^{17/3}M^{7/2}}{Q^2} + \frac{Q^4L^{8/3}M^{13/2}t^2}{N^{3/2}} +
  \frac{N^{5/2}L^{17/3}M^{3/2}}{Q^3t^2} +
  \frac{N^{5/2}L^{20/3}M^{5/2}}{Q^{4}t}\right)^{1/2}
  \\
  \ll r^{1/2}N^{3/4+\varepsilon}M^{1/2}L^{3/4}\left(\frac{t}{K^{1/2}}+
  K^{3/4} + \frac{K^{5/4}}{t^{1/2}}\right)
  + r^{1/2}N^{1+\varepsilon}\frac{L^{1/2}K}{M^{1/4}t}.
\end{multline}


\subsubsection{Case b}
By the same arguments, we obtain
\begin{align*}
  \Omega_{\mathbf{b},1}\ll \sup_{\substack{1\ll\tilde{N}_2\ll N_2\\ \textup{dyadic}}}
  \frac{N^\varepsilon q_1^4rN_0M^{5/2}M_1^{1/2}L^{8/3}\tilde{N}_2C(\tilde{N}_2)}{n_1^3R^2}
  (R+M_1),
\end{align*}
where $M_1\ll\frac{M^2R^2t^\varepsilon}{N}$ and $R\asymp\frac{NLX}{MQt}$.
So we see that
\begin{align*}
   \Omega_{\mathbf{b},1}\ll \frac{N^\varepsilon q_1^3r^2}{n_1^2}
  \left(\frac{N^{3/2}L^{14/3}M^{5/2}}{Q^2t^2} +
  \frac{N^{3/2}L^{17/3}M^{7/2}}{Q^{3}t^3}\right),
\end{align*}
which contributes \eqref{eqn:led to estimate case b} at most
\begin{multline}\label{eqn:contribiton case b}
  \ll \frac{N^{1+\varepsilon}X^{1/2}r^{1/2}}{M^{5/2}L^{4/3}Q^{1/2}}
  \sum_{n_1\ll RMr}\frac{1}{n_1^{1-\theta_3}}
  \sum_{\frac{n_1}{(n_1,r)}|q_1|n_1^\infty}\frac{1}{q_1^{1/2}}
  \left(\frac{N^{3/2}L^{14/3}M^{5/2}}{Q^2t^2} +
  \frac{N^{3/2}L^{17/3}M^{7/2}}{Q^{3}t^3}\right)^{1/2}
  \\
  \ll \frac{N^{1+\varepsilon}r^{1/2}L^{1/4}K^{3/4}}{M^{1/2}t}
  + \frac{N^{3/4+\varepsilon}r^{1/2}L^{1/2}M^{1/4}K}{t^{3/2}}.
\end{multline}

\subsubsection{Case c}
Similarly, by the same treatment and the results in \S 5.3, we have
\begin{align*}
  \Omega_{\mathbf{c},1}\ll
  \frac{N^\varepsilon q_1^4rN_0M^{5/2}M_1^{1/2}L^{3}\tilde{N}_2}{n_1^3R^2}
  (R+M_1),
\end{align*}
where $N_0\ll \frac{R^3M^3r}{LN}t^\varepsilon$, $\tilde{N}_2=\frac{MrR^2n_1}{q_1N_0}$
and $M_1=\frac{R^2M^2t^2}{N}$.
It is easy to see that
\[
  \Omega_{\mathbf{c},1}\ll \frac{q_1^{3}r^2}{n_1^2}\left(\frac{R^3M^{13/2}t^3L^3}{N^{3/2}}
  + \frac{R^2M^{9/2}tL^3}{N^{1/2}}\right).
\]
Notice that $X\ll\frac{MRQ}{NL}t^\varepsilon$ now.
Hence, the contribution from $\Omega_{\mathbf{c},1}$ to \eqref{eqn:led to estimate c} is
\begin{multline}\label{eqn:contribiton case c}
  \ll \frac{N^{5/4+\varepsilon}X r^{1/2}}{ M^{5/2}LQ }  \frac{1}{t^{1/2}}
  \sum_{\eta_1=\pm 1} \sum_{n_1 \leq Rr} \frac{1}{n_1^{1-\theta_3}}
   \sum_{\frac{n_1}{(n_1,r)}\mid q_1\mid (rn_1)^\infty} \frac{1}{q_1^{1/2}}
   \left(\frac{R^3M^{13/2}t^3L^3}{N^{3/2}}
  + \frac{R^2M^{9/2}tL^3}{N^{1/2}}\right)^{1/2}
  \\
  \ll \frac{r^{1/2}N^{3/4}M^{1/2}L^{3/4}t}{K^{5/4}} + \frac{r^{1/2}N L^{1/2} }{M^{1/4}K}.
\end{multline}
\subsection{$n_2\equiv 0\; (\mod M)$, $n_2\neq0$}
Denote the contribution of this part to $\Omega$ by $\Omega_{2}$.
By the congruence condition $q_2'\bar{\alpha}-q_2\bar{\alpha'}\equiv-n_2 \; (\mod M)$,
we have $\alpha'\equiv \bar{q_2'}q_2\alpha \; (\mod M)$.
Hence,
\begin{align*}
  \mathfrak{C}(n_2)\ll |\mathfrak{C}_1(n_2)||\mathfrak{C}_2(n_2)||\mathfrak{C}_3(n_2)|,
\end{align*}
where $\mathfrak{C}_2(n_2)$ and $\mathfrak{C}_3(n_2)$ are defined as in \S 7.1,
and
\begin{align*} 
  \mathfrak{C}_1(n_2)& = \;\sideset{}{^\star}\sum_{b \bmod M }
  \;\sideset{}{^\star}\sum_{b' \bmod M }
  \left(\sum_{\substack{u \bmod M  \\ u\neq b}}\bar{\chi}(u)e\left(\frac{m\overline{q_2^2\ell(b-u)}}{M}\right)\right)
  \\
  & \quad \cdot \left(\sum_{\substack{u' \bmod M  \\ u'\neq b'}}\chi(u')e\left(-\frac{m'\overline{q_2'^2\ell'(b'-u')}}{M}\right)\right)
  \left( \; \sideset{}{^\star}\sum_{\substack{\alpha  \bmod M }}
  e\left(\frac{\alpha\overline{bq_2^2}-\alpha q_2\overline{b'q_2'^3}}{M}\right)\right)
\end{align*}
Note that the innermost $\alpha$-sum is a Ramanujan sum. We get
\begin{align*} 
  \mathfrak{C}_1(n_2)
  & \ll M \bigg| \; \sideset{}{^\star}\sum_{b \bmod M }
  \bigg(\sum_{\substack{u\bmod M \\ u\neq b}}\bar{\chi}(u)e\left(\frac{n\overline{q_2^2\ell(b-u)}}{M}\right)\bigg)
  \\ & \hskip 90pt \cdot
  \bigg(\sum_{\substack{u'\bmod M \\ u'\neq bq_2^3\bar{q_2'}^3}}\chi(u') e\left(\frac{n'\overline{q_2'^2\ell'(bq_2^3\bar{q_2'}^3-u')}}{M}\right)\bigg)\bigg|
  \\
  & \qquad
  + \bigg| \; \sideset{}{^\star}\sum_{b \bmod M }
  \sideset{}{^\star}\sum_{\substack{b' \bmod M \\ b' \not\equiv  bq_2^3\bar{q_2'}^3 \bmod M }}
  \bigg(\sum_{\substack{u\bmod M \\ u\neq b}}\bar{\chi}(u)e\left(\frac{n\overline{q_2^2\ell(b-u)}}{M}\right)\bigg)
  \\ & \hskip 90pt \cdot
  \bigg(\sum_{\substack{u'\bmod M \\ u'\neq b' }}\chi(u') e\left(\frac{n'\overline{q_2'^2\ell'(b'-u')}}{M}\right)\bigg)\bigg|.
\end{align*}
As in \cite[\S 6.2]{Sharma2019}, there is a square root cancellation in the sum over $u$ and $u'$,
so we arrive at
\[
  \mathfrak{C}_1(n_2)\ll M^3.
\]
Therefore, by the same treatment as in \S 7.1 together with the condition $n_2\equiv 0 \; (\bmod M)$,
we can get a better result than that in \S 7.1.

Combining the above argument together with , the contribution of the non-zero frequencies can be dominated by
\begin{multline}\label{eqn:contribution non-zero frequencies}
  \ll r^{1/2}N^{3/4+\varepsilon}M^{1/2}L^{3/4}\left(\frac{t}{K^{1/2}}+
  K^{3/4} + \frac{K^{5/4}}{t^{1/2}}\right)
  + \frac{r^{1/2}N^{1+\varepsilon}L^{1/2}}{M^{1/4}}\left(\frac{K}{t}+\frac{1}{K}\right).
\end{multline}

\section{Proof of Proposition 3.1}\label{section:proof of Proposition 3.1}

Now we are ready to give an upper bound for $S_{11}^\pm(N,X,R)$ when $(r,M)=1$.
By \eqref{eqn:contribution zero frequency} and \eqref{eqn:contribution non-zero frequencies},
we get
\begin{multline*}
  S_{11}^\pm(N,X,R)\ll r^{1/2}  N^{1/2+\varepsilon} L^{1/2} M^{5/4} (K^{3/2} + t)
   +
    \frac{r^{1/2}N^{3/4+\varepsilon}  M^{3/4} }
    { L^{1/4}}\left(K^{3/4}+\frac{K^{5/4}}{t^{1/2}}\right)
   \\
   +r^{1/2}N^{3/4+\varepsilon}M^{1/2}L^{3/4}\left(\frac{t}{K^{1/2}}+
  K^{3/4} + \frac{K^{5/4}}{t^{1/2}}\right)
  + \frac{r^{1/2}N^{1+\varepsilon}L^{1/2}}{M^{1/4}}\left(\frac{K}{t}+\frac{1}{K}\right).
\end{multline*}
By taking
\begin{align*}
  L=M^{1/4},\quad K=t^{4/5},
\end{align*}
we deduce that
\begin{align*}
  S_{11}^\pm(N,X,R) \ll N^{1/2+\varepsilon}M^{3/2-1/16}t^{3/2-3/20},
\end{align*}
provided that $N\ll (Mt)^{3+\varepsilon}/r^2$ and $r\ll M^{1/8}t^{3/10}$.

As we point out in \S \ref{sec:reduction}, all the other cases
(such as $(r,M)>1$, $S_{12}^\pm(N,X,R)$, $S_{13}^\pm(N,X,R)$, $S_2(N)$, $S_3(N)$) are similar and in fact
easier. Hence, we finally prove Proposition \ref{reduction prop}.


\section*{Acknowledgements}
We would like to thank Yongxiao Lin and  Qingfeng Sun for helpful discussions and comments.
We are grateful to the referee for his/her very helpful comments and suggestions.



\begin{thebibliography}{10}

\bibitem{Aggarwal}
K. Aggarwal,
A new subconvex bound for $\GL(3)$ $L$-functions in the $t$-aspect.
\emph{Int. J. Number Theory}  17 (2021), no. 5, 1111--1138. 




\bibitem{blomer2012subconvexity}
V. Blomer,
Subconvexity for twisted $L$-functions on $\GL(3)$.
\emph{Amer. J. Math.} 134 (2012), no. 5, 1385--1421.

\bibitem{BlomerHarcos}
V. Blomer and G. Harcos,
Hybrid bounds for twisted $L$-functions.
\emph{J. Reine Angew. Math.} 621 (2008), 53--79.



\bibitem{BlomerKhanYoung}
V.~Blomer, R.~Khan, and M.~Young,
\newblock Distribution of mass of holomorphic cusp forms.
\newblock {\em Duke Math. J.} 162 (2013), no. 14, 2609--2644.



\bibitem{Burgess}
D. Burgess,
\newblock On character sums and {$L$}-series. {II}.
\newblock \emph{Proc. London Math. Soc. (3)}, 1963, 13:524--536.


\bibitem{FanSun}
Y. Fan and Q. Sun,
A Bessel $\delta$-method and hybrid bounds for $\GL_2$.
\emph{ArXiv preprint} (2020), arXiv:2008.09871.
%
%
%

\bibitem{goldfeld2006automorphic}
D. Goldfeld,
\emph{Automorphic forms and $L$-functions for the group $\GL(n,\mathbb R)$}. With an appendix by Kevin A. Broughan. Cambridge Studies in Advanced Mathematics, 99. Cambridge University Press, Cambridge, 2006. xiv+493 pp.

\bibitem{goldfeld2006voronoi}
D. Goldfeld and X. Li,
\newblock Voronoi formulas on {${\rm GL}(n)$}.
\emph{Int. Math. Res. Not.} 2006, Art. ID 86295, 25 pp.

%

\bibitem{Heath-Brown}
D. R. Heath-Brown,
Hybrid bounds for Dirichlet $L$-functions.
\emph{Invent. Math.} 47 (1978), no. 2, 149--170. 

\bibitem{HolowinskyNelson}R. Holowinsky and P. Nelson,
Subconvex bounds on $\GL_3$ via degeneration to frequency zero.
\emph{Math. Ann.} 372 (2018), no. 1-2, 299--319.


\bibitem{Huang2019}
B. Huang,
Hybrid subconvexity bounds for twisted $L$-functions on GL(3). \emph{Sci. China Math.} 64 (2021), no. 3, 443--478. 



\bibitem{Huang}
B. Huang,
On the Rankin--Selberg problem.
 \emph{Math. Ann.} 381 (2021), no. 3-4, 1217--1251.

\bibitem{Huanguniform}
B. Huang,
Uniform bounds for GL(3)$\times$GL(2) $L$-functions.
\emph{ArXiv preprint} (2021), arXiv:2104.13025.








\bibitem{IwaniecKowalski2004analytic}
H.~Iwaniec and E.~Kowalski,
\newblock {\em Analytic number theory}, volume~53 of {\em American Mathematical
  Society Colloquium Publications}.
\newblock American Mathematical Society, Providence, RI, 2004.

\bibitem{Kim2003}
H. Kim,
Functoriality for the exterior square of $\GL_4$ and the symmetric fourth of $\GL_2$,
\emph{J. Amer. Math. Soc.} 16 (2003), no. 1, 139--183.
With appendix 1 by Ramakrishnan and appendix 2 by Kim and Sarnak.





\bibitem{KPY}
E. K{\i}ral, I. Petrow, and M. Young,
Oscillatory integrals with uniformity in parameters.
\emph{J. Th\'eor. Nombres Bordeaux} 31 (2019), no. 1, 145--159.

\bibitem{KMV2002}
E. Kowalski, P. Michel, and J. VanderKam,
Rankin--Selberg $L$-functions in the level aspect.
\emph{Duke Math. J.} 114 (2002), no. 1, 123--191.


\bibitem{li2011bounds}
X.~Li,
\newblock Bounds for {${\rm GL}(3)\times {\rm GL}(2)$} {$L$}-functions and
  {${\rm GL}(3)$} {$L$}-functions.
\newblock {\em Ann. of Math. (2)}, 173(1):301--336, 2011.





\bibitem{Lin}
Y. Lin,
Bounds for twists of $\GL(3)$ $L$-functions.
 \emph{J. Eur. Math. Soc. (JEMS)} 23 (2021), no. 6, 1899--1924.

\bibitem{LMS}
Y. Lin, Ph. Michel and W. Sawin,
Algebraic twists of $\GL_3\times \GL_2$ $L$-functions.
\emph{ArXiv preprint} (2019), arXiv:1912.09473.

\bibitem{LinSun}
Y. Lin and Q. Sun,
Analytic Twists of $\GL_3 \times \GL_2$ Automorphic Forms.
\emph{Int. Math. Res. Not. IMRN} 2021, no. 19, 15143--15208.




%


%
%


%


\bibitem{MichelVenkatesh}
Ph. Michel and A. Venkatesh,
The subconvexity problem for $\GL_2$.
\emph{Publ. Math. Inst. Hautes \'Etudes Sci.} No. 111 (2010), 171--271.

\bibitem{MillerSchmid2006automorphic}
S. Miller and W. Schmid,
\newblock Automorphic distributions, {$L$}-functions, and {V}oronoi summation
  for {${\rm GL}(3)$}.
\newblock \emph{Ann. of Math. (2)}, 2006, 164(2):423--488.



%



\bibitem{Munshi2015circleIII}
R. Munshi,
The circle method and bounds for $L$-functions--III: $t$-aspect subconvexity for $\GL(3)$ $L$-functions.
\emph{J. Amer. Math. Soc.} 28 (2015), no. 4, 913--938.

\bibitem{Munshi2015circleIV}
R. Munshi,
The circle method and bounds for $L$-functions--IV: Subconvexity for twists of $\GL(3)$ $L$-functions. \emph{Ann. of Math. (2)} 182 (2015), no. 2, 617--672.



\bibitem{Munshi2018}
R. Munshi,
Subconvexity for $\rm GL (3)\times GL (2)$ $L$-functions in $t$-aspect.
\emph{ArXiv preprint} (2018), arXiv:1810.00539.

\bibitem{PetrowYoung}
I. Petrow and M. Young,
The Weyl bound for Dirichlet $L$-functions of cube-free conductor.
\emph{Ann. of Math. (2)} 192 (2020), no. 2, 437--486. 

\bibitem{PetrowYoung2019}
I. Petrow and M. Young,
The fourth moment of Dirichlet L-functions along a coset and the Weyl bound.
\emph{ArXiv preprint} (2019), arXiv:1908.10346.

\bibitem{Sharma2019}
P. Sharma,
Subconvexity for $\GL(3)\times \GL(2)$ twists in level aspect.
\emph{ArXiv preprint} (2019), arXiv:1906.09493.


\bibitem{Weyl}
H. Weyl,
Zur absch\"atzung von $\zeta(1+it)$.
\emph{Math. Z.}, 10 (1921), 88--101.

\end{thebibliography}
\end{document}